\documentclass[BCOR12mm,DIV12,abstracton,headinclude]{scrartcl}
\usepackage{scrpage2}
\usepackage{amsfonts}
\usepackage{amsthm}
\usepackage[nottoc]{tocbibind}
\usepackage{enumerate}

\usepackage[german,british]{babel}
\usepackage[all]{xy}
\usepackage[leqno]{amsmath}
\usepackage{amssymb}

\title{Auslander-Reiten theory for simply connected differential graded algebras}
\author{Karsten Schmidt}

\newtheorem{lem}{Lemma}[section]
\newtheorem{prop}[lem]{Proposition}
\newtheorem{cor}[lem]{Corollary}
\newtheorem{thm}[lem]{Theorem}
\newtheorem{conj}[lem]{Conjecture}
\newtheorem*{introprop1}{Proposition~\ref{prop:ampl}}
\newtheorem*{introthm}{Theorem~\ref{thm:noc}}
\newtheorem*{introprop2}{Proposition~\ref{prop:func}}

\theoremstyle{remark}
\newtheorem{rem}[lem]{Remark}

\theoremstyle{definition}
\newtheorem{exm}[lem]{Example}
\newtheorem{defn}[lem]{Definition}
\newtheorem*{Ackn}{Acknowledgements}

\renewcommand{\mod}{\operatorname{mod}\nolimits}

\newcommand{\rad}{\operatorname{rad}\nolimits}

\newcommand{\id}{\operatorname{id}\nolimits}

\newcommand{\Mod}{\operatorname{Mod}\nolimits}
\newcommand{\End}{\operatorname{End}\nolimits}
\newcommand{\Hom}{\operatorname{Hom}\nolimits}
\newcommand{\RHom}{\operatorname{RHom}\nolimits}
\newcommand{\LHom}{\operatorname{LHom}\nolimits}
\newcommand{\Aut}{\operatorname{Aut}\nolimits}

\newcommand{\Ho}{\operatorname{H}\nolimits}
\renewcommand{\Im}{\operatorname{Im}\nolimits}
\newcommand{\Ker}{\operatorname{Ker}\nolimits}

\renewcommand{\dim}{\operatorname{dim}_k\nolimits}
\newcommand{\dimt}{\operatorname{dim}\nolimits}

\newcommand{\Ext}{\operatorname{Ext}\nolimits}

\newcommand{\Ab}{\mathrm{Ab}}
\newcommand{\DGAlg}{\mathrm{dgAlg_k}}
\newcommand{\CDGAlg}{\mathrm{cdgAlg_k}}
\newcommand{\GAlg}{\mathrm{gAlg_k}}
\newcommand{\Ob}{\mathrm{Ob}}

\newcommand{\amp}{\operatorname{amp}\nolimits}
\newcommand{\lev}{\operatorname{level}\nolimits}

\newcommand{\thick}{\operatorname{thick}\nolimits}
\newcommand{\op}{\mathrm{op}}

\newcommand{\fp}{\operatorname{Fp}\nolimits}

\def\a{\alpha}
\def\b{\beta}

\def\e{\varepsilon}

\def\g{\gamma}
\def\p{\varphi}

\def\t{\tau}

\def\m{\mu}
\def\n{\nu}

\def\la{\lambda}

\def\Si{\Sigma}

\def\C{{\mathcal C}}
\def\D{{\mathbf D}}
\def\E{{\mathcal E}}

\def\H{{\mathcal H}}

\def\M{{\mathcal M}}
\def\N{{\mathcal N}}

\def\X{{\mathcal X}}

\def\T{{\mathcal T}}
\def\U{{\mathcal U}}

\begin{document}

\pagestyle{scrheadings}
\chead{Auslander-Reiten theory for simply connected differential graded algebras}
\ohead{\pagemark}
\ihead{\headmark}
\cfoot{}

\flushbottom

\author{}
\date{}

\pagenumbering{roman}
\maketitle
\thispagestyle{empty}
\vspace{-2cm}
\begin{center}
{\Large\bf Dissertation}\\vorgelegt von\\[0.5cm]{\Large Karsten Schmidt}\\[0.5cm]an der Fakult\"at f\"ur Elektrotechnik, Informatik und Mathematik \\der Universit\"at Paderborn\\[0.5cm]im November 2007
\end{center}
\vspace{1cm}
\begin{abstract}
In \cite{J1} and \cite{J2} J\o rgensen introduced the Auslander-Reiten quiver of a simply connected Poincar\'e duality space. He showed that its components are of the form $\mathbb Z A_\infty$ and that the Auslander-Reiten quiver of a $d$-dimensional sphere consists of $d-1$ such components. In this thesis we show that this is the only case where finitely many components appear. More precisely, we construct families of modules, where for each family, each module lies in a different component. Depending on the cohomology dimensions of the differential graded algebras which appear, this is either a discrete family or an $n$-parameter family for all $n$.

%
%
%
%
%
\end{abstract}

\begin{otherlanguage}{german}
\begin{abstract}
J\o rgensen hat in seinen Artikeln \cite{J1} und \cite{J2} den Auslander-Reiten K\"ocher eines einfach zusammenh\"angenden topologischen Raumes mit Poincar\'e Dualit\"at eingef\"uhrt. Er zeigt, dass die Komponenten von der Form $\mathbb Z A_\infty$ sind, und dass der Auslander-Reiten K\"ocher einer Sph\"are der Dimension $d$ aus $d-1$ solcher Komponenten besteht. In dieser Arbeit zeigen wir, dass dies der einzige Fall ist, in dem nur endlich viele Komponenten auftreten. Wir konstruieren in Abh\"angigkeit von der Dimension der Kohomologie der zugeh\"origen differenziell graduierten Algebren diskrete beziehungsweise beliebige Parameter Familien von Moduln, wobei alle diese Moduln in unterschiedlichen Komponenten liegen.
\end{abstract}
\end{otherlanguage}

\newpage

\thispagestyle{plain}
\tableofcontents

\newpage

\pagenumbering{arabic}

\section{Introduction: Motivation and Summary of the main results}

\subsection{Motivation}
Algebraic topology is concerned with the study of algebraic invariants of topological spaces. Two important concepts in this subject are homotopy groups and (co)homology groups. Let us restrict to cohomology here. Given a topological space and a field $k$ one method of calculating cohomology is to consider the singular cochain complex $C^*(X;k)$ with coefficients in $k$
 $$\cdots\overset d\to C^{n-1}(X;k)\overset d\to C^{n}(X;k)\overset d\to C^{n+1}(X;k)\overset d\to\cdots .$$
Taking cohomology in position $n$ gives the $n$-th cohomology group $\Ho^n(X;k)$ of the space $X$. Considering all cohomology groups together and keeping track of the grading gives a graded vectorspace
$$\Ho^*(X;k)=\bigoplus_{n\in\mathbb Z}\Ho^n(X;k).$$
Cohomology has in comparison with homology a richer structure: namely the cup product defines a multiplicative structure on $\Ho^*(X;k)$ and we get a graded algebra. In general the singular cochain complex includes more than just the information about the cohomology so it is natural to consider the singular cochain complex itself as algebraic object. The cup product is already defined on the singular cochain complex and makes it into a graded algebra. Additional structure is provided by the differential $d$ mapping from degree $n$ to degree $n+1$. It satisfies $d^2=0$ and behaves well with the multiplication, namely it satisfies the Leibniz rule
$$d(ab)=d(a)b+(-1)^{\deg a}ad(b).$$
A graded algebra together with a differential satisfying the Leibniz rule is called a differential graded algebra. These objects, especially those differential graded algebras appearing as the singular cochain differential graded algebra of a simply connected topological space $X$ are the objects of study in this thesis.

In order to understand algebraic objects, in representation theory one studies representations of these objects. These are structure preserving morphisms from the given object to an object with the same algebraic structure that comes from linear algebra. For example a representation of a group $G$ over a field $k$ is a group homomorphism from $G$ to the group of $k$-linear automorphisms of a $k$-vectorspace $V$.
$$G\to\Aut_k(V)$$
Similarly, a representation of a $k$-algebra $A$ is a $k$-algebra homomorphism from $A$ to the $k$-algebra of $k$-linear endomorphisms of a $k$-vectorspace $V$. 
$$A\to\End_k(V)$$
Equivalently, one may describe a representation of an object $A$ by giving a linear object $V$ with an action of $A$ that satisfies the corresponding properties. One calls the object $V$ together with the action of the object $A$ an $A$-module.

In this thesis we will consider differential graded modules over differential graded algebras. They are representations of the differential graded algebra $A$ in the following sense. Given a complex of $k$-vector spaces
$$V\colon\quad\cdots\to V^{n-1}\to V^{n}\to V^{n+1}\to\cdots$$
one can construct the endomorphism differential graded algebra $\E nd(V)$. A representation is then a morphism of differential graded algebras
$$A\to\E nd_k(V).$$

In order to study homological properties of modules like $\Ext$-groups, homological dimensions, etc. there is a natural environment, namely the derived category. Many algebraic invariants of a ring like the centre, Hochschild cohomology, the Grothendieck group, etc. are also invariants of the derived category. The derived category of a ring $R$ is usually constructed in three steps: One first considers the category of complexes of $R$-modules; then one forms the homotopy category by factoring out the homotopy relation for morphisms of complexes; and the third step is to get the derived category by formally inverting all quasi-isomorphisms, i.e.\ isomorphisms in cohomology. The homotopy category and the derived category are in general no longer abelian categories, but they carry the structure of a triangulated category, i.e.\ there is a class of distinguished triangles
$$A\to B\to C\to\Sigma A$$
that satisfy several axioms.

The construction of the derived category generalises to our case of a differential graded algebra. In this case the category of differential graded modules over $A$  already corresponds to the category of complexes. The construction of the homotopy category and the derived category then goes through as before.

The Auslander-Reiten quiver of an additive category $\C$ gives a visualisation of the indecomposable objects and irreducible morphisms between them as a directed graph. In particular, it makes sense to consider the Auslander-Reiten quiver if every object decomposes uniquely as the direct sum of indecomposable objects, i.e.\ $\C$ is a Krull-Remak-Schmidt category. If $\C$ is a triangulated Krull-Remak-Schmidt category that has Auslander-Reiten triangles, then one gets the irreducible morphisms from these triangles and the quiver is equipped with the richer structure of a stable valued translation quiver.

The aim is now to study Auslander-Reiten theory over differential graded algebras.

\subsection{J\o rgensen's results}
Our work is based upon the results of J\o rgensen from \cite{J1} and \cite{J2}, which we now summarise. The precise statements are Theorem~\ref{thm:poincare}, Theorem~\ref{thm:ZA} and Example~\ref{exm:sphere} in Section~3. Given a simply connected topological space $X$, J\o rgensen considers inside the derived category $\D(X):=\D(C^*(X;k))$ of the singular cochain differential graded algebra the full subcategory of compact objects $\D^c(X)$, where the compactness of an object is a certain finiteness condition. For example, over a ring $R$ all finitely generated $R$-modules are compact in the category of all $R$-modules $\Mod R$. He shows that the existence of Auslander-Reiten triangles in $\D^c(X)$ is equivalent to the condition that $X$ is a Poincar\'e duality space. This is equivalent to say that $C^*(X;k)$ is a Gorenstein differential graded algebra (cf.\ \cite{FHT2}). The corresponding Auslander-Reiten quiver is the combinatorial object that has as vertices the isomorphism classes of indecomposable objects and as arrows the irreducible maps between them. It is a weak homotopy invariant of the space $X$ since the singular cochain differential graded algebra of $X$ has this property. J\o rgensen shows that the components of the Auslander-Reiten quiver that appear are always of the form $\mathbb ZA_\infty$:
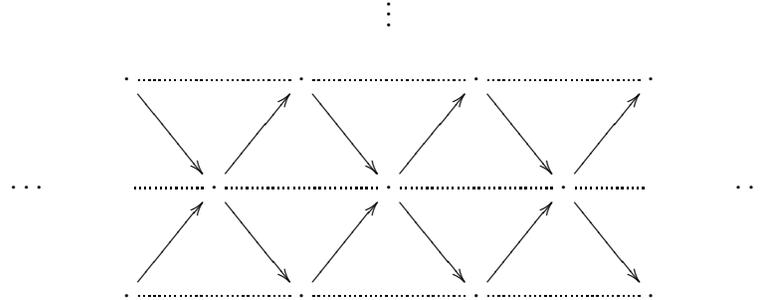
$$\xymatrixrowsep{1pc}\xymatrix{
&&&& \vdots \\
&\cdot \ar[rdd] \ar@{.}[rr] & & \cdot \ar[rdd] \ar@{.}[rr] & & \cdot \ar[rdd] \ar@{.}[rr] & & \cdot \\
\\
\cdots & \ar@{.}[r] & \cdot \ar[rdd] \ar[ruu] \ar@{.}[rr] & & \cdot \ar[rdd] \ar[ruu] \ar@{.}[rr] & & \cdot \ar[rdd] \ar[ruu] \ar@{.}[r] & & \cdots \\
\\
&\cdot \ar[ruu] \ar@{.}[rr] &  & \cdot \ar[ruu] \ar@{.}[rr] & & \cdot \ar[ruu] \ar@{.}[rr] & & \cdot \\
}$$
He further computes that the Auslander-Reiten quiver of the $d$ dimensional sphere $S^d$ consists of $d-1$ such components. A natural question is then whether the number of components for a Poincar\'e duality space of Poincar\'e dimension $d$ will in general be $d-1$. Our main result will give a negative answer to this question.

\subsection{Organisation of the paper and the main results}
In the following we describe how the paper is organised and state the main results:

In the second section we consider a differential graded $k$-algebra $A$ with degreewise finite dimensional cohomology and summarise results concerning the full subcategory of compact objects in the derived category of $A$. We recall that it is a Krull-Remak-Schmidt category and the existence of Auslander-Reiten triangles is equivalent to having Serre duality.

In the third section we consider simply connected differential graded algebras. We explain the construction of a so-called free model that is inspired by rational homotopy theory and the construction of minimal semi-free resolutions that are important for the proofs. We state J\o rgensen's results from \cite{J1} and \cite{J2}. For his result about the shape of the Auslander-Reiten components we are able to remove the assumption that the characteristic of the ground field is zero. This is possible since we give the following cochain analogue of the amplitude inequalities in \cite{J3}.

\begin{introprop1}
Let $A$ be a simply connected differential graded algebra of finite type, $0\neq M\in\D^c(A)$, and $0\neq X\in\D(A^\op)$ with $\dim\Ho^*X<\infty$ .
Then
\begin{enumerate}
\item $\inf\Ho^*(M\otimes_A^L X)=\inf\Ho^*M+\inf\Ho^*X$.
\item $\sup\Ho^*(M\otimes_A^L X)\geq\inf\Ho^*M+\sup\Ho^*X$.
\item $\amp(M\otimes_A^L X)\geq\amp X$.
\end{enumerate}
In particular, $\amp M\geq\amp A$.
\end{introprop1}

%

In the fourth section we are concerned with the number of Auslander-Reiten components for a simply connected Gorenstein differential graded algebra $A$ of finite type. If $\dim \Ho^*A=1$, then the Auslander-Reiten quiver consists of countably many components each of them containing just a single vertex. If $\dim \Ho^*A=2$, then we are in the case of the spheres where J\o rgensen has computed that the Auslander-Reiten quiver consists of $\sup\{i\mid\Ho^iA\neq 0\}-1$ components of the form $\mathbb ZA_\infty$. Our main result shows that these cases are quite special:

\begin{introthm}
Let $A$ be a simply connected Gorenstein differential graded algebra of finite type. Then the following hold.
\begin{enumerate}
\item The Auslander-Reiten quiver of $A$ has finitely many components if and only if $\dim \Ho^*A=2$. In this case the number of components equals $\sup\{i\mid\Ho^iA\neq 0\}-1$.
\item If $\dim \Ho^eA\geq 2$ for some $e$, then there is an $n$-parameter family of Auslander-Reiten components for each $n\in\mathbb N$. In fact there are objects, each lying in a different component, that can be parametrised by $\mathbb P^1(k)^n$.
\end{enumerate}
\end{introthm}

To prove this theorem we construct objects which lie in different components. For those objects we can inductively describe the endomorphism ring (see Proposition~\ref{prop:local} for the precise statement).

In the fifth section we give, based on the functorial approach to Auslander-Reiten sequences that Auslander describes in \cite{A2}, a different interpretation of the number of components in the Auslander-Reiten quiver of a triangulated Krull-Remak-Schmidt category $\T$ having Auslander-Reiten triangles. This is in terms of certain equivalence classes of finitely presented simple functors in the functor category $\widehat\T=\fp(\T^\op,\Ab)$. We prove

\sloppy
\begin{introprop2}
Let $\T$ be a triangulated Krull-Remak-Schmidt category having Auslander-Reiten triangles. Then there is a natural one-to-one correspondence between the Auslander-Reiten components of $\T$ and the equivalence classes of simple objects in $\widehat\T$.
\end{introprop2}
\fussy

In the sixth section we discuss some open questions and conjectures.

The remaining sections are appendices about the derived category of a differential graded algebra, Auslander-Reiten theory and a little topological background.

\subsection{Notation and terminology}
 Throughout this thesis the differential graded objects will usually be denoted in the cohomological notation, i.e. the differential increases degree by $1$. Furthermore $k$ will denote a field and $A$ a differential graded $k$-algebra. The opposite differential graded algebra will be denoted by $A^\op$. Further $\mathbf C(A)$ will be the category of differential graded $A$-modules, $\mathbf H(A)$ the corresponding homotopy category, and $\D(A)$ the derived category (for further explanations look at \cite{Ke} or the Appendix~A). We write superscript $f$ for the corresponding subcategories consisting of all objects $X$ with the $n$-th cohomology $\Ho^nX$ finitely generated over $k$ for all $n\in\mathbb{Z}$, superscript $b$ for the corresponding subcategories consisting of all objects $X$ whose underlying complexes are bounded in cohomology, and superscript $c$ for the corresponding subcategories consisting of all compact objects. We will denote the shift of (differential) graded objects by $\Sigma$. If we consider just the underlying graded structure of a differential graded object $X$, then we will sometimes write $X^\natural$. Furthermore, let $D:=\Hom_k(-,k)$ be the $k$-duality functor. Given a complex $X$ the cohomology $\Ho^*X$ is the $\mathbb Z$-graded object that consists in degree $n$ of the $n$-th cohomology $\Ho^nX$. We sometimes denote the graded object $\Ho^*\RHom_A(X,Y)$ by $\Ext^*_A(X,Y)$. We also define 
$$\sup X:=\sup\{i\in\mathbb Z\mid X^i\neq 0\}~\text{ and }~\inf X:=\inf\{i\in\mathbb Z\mid X^i\neq 0\}.$$
For a given category $\C$ and two objects $X,Y\in\Ob(\C)$ the set of morphisms $\Hom_\C(X,Y)$ will sometimes be denoted by $\C(X,Y)$ or just $(X,Y)$ for short. Given a triangulated category $\T$ and a set of objects $S$ in $\T$, we denote by $\langle S\rangle_{\rm thick}$ the {\em thick} subcategory generated by $S$, i.e.\ the smallest triangulated subcategory of $\T$ that contains $S$ and is closed under direct summands. Similarly, we denote by $\langle S\rangle_{\rm loc}$ the {\em localising} subcategory generated by $S$, i.e.\ the smallest triangulated subcategory of $\T$ that contains $S$ and is closed under taking coproducts. For us a triangulated subcategory will be strict (i.e.\ closed under isomorphisms) by definition. We say that a $k$-linear category is {\em $\Hom$-finite} if all homomorphism spaces are finite dimensional.

Further notation is explained in the appendices.

\begin{Ackn} I would like to thank my advisor Henning Krause for his support of my dissertation. I also thank the representation theory group in Paderborn and their various guests for their mathematical and non-mathematical support. In particular, I am grateful to Dave Benson and Srikanth Iyengar for some helpful discussions. I also thank Kristian Br\"uning and Andrew Hubery for some useful comments on a preliminary version of this thesis.
\end{Ackn}
\newpage

\section{Preliminaries}



Let $k$ be a field and $A$ a differential graded $k$-algebra. Before we get to the simply connected differential graded algebras in Section~3, we consider differential graded algebras with degreewise finite dimensional cohomology. We recall some facts related to the existence of Auslander-Reiten triangles in $\D^c(A)$, i.e.\ the full subcategory of the derived category $\D(A)$ consisting of the compact objects. For the convenience of the reader some proofs are included.
  
\begin{lem}\label{lem:duality}
Let $A$ be a differential graded $k$-algebra, $X,Y$ be objects in $\mathbf D(A)$, and in addition let $X$ be compact. Then there is an isomorphism that is natural in $X$ and $Y$
$$D\Hom_{\D(A)}(X,Y)\cong\Hom_{\D(A)}(Y,X\otimes_A^LDA).$$
\end{lem}
\begin{proof}
The same calculation as in the ordinary algebra case (cf.\ \cite[Example~1]{K1}) gives the desired result.
\end{proof}

\begin{lem}\label{lem:dual} Let $A$ be a differential graded algebra with degreewise finite dimensional cohomology. Then
\begin{enumerate}
\item $\Hom_{\D(A)}(X,Y)$ is a finite dimensional $k$-vectorspace for every pair of compact objects $X,Y$.
\item $\End_{\D(A)}(X)$ is local for every indecomposable compact object $X$.
\end{enumerate}
In particular, $\D^c(A)$ is a $\Hom$-finite $k$-linear Krull-Remak-Schmidt category.
\begin{proof}
(1) Consider the full subcategory $\M$ in $\mathbf D(A)$ consisting of all objects $X$ such that $\Hom_{\D(A)}(X,\Si^nA)$ is finite dimensional over $k$ for all $n\in\mathbb{Z}$. The category $\M$ is a thick triangulated subcategory of $\D(A)$ containing $A_A$. Here one uses the fact that $\Hom_{\D(A)}(A,\Si^nA)\cong\Ho^nA$ is finitely generated. Hence $\M$ contains $\langle A\rangle_{\rm thick}=\mathbf D^c(A)$. Now consider the full subcategory $\N$ in $\mathbf D(A)$ consisting of all objects $Y$  such that $\Hom_{\D(A)}(X,Y)$ is finite dimensional over $k$ for all $X\in\M$. The category $\N$ is again a thick triangulated subcategory of $\D(A)$ containing $A_A$, hence $\N$ contains  $\langle A\rangle_{\rm thick}=\D^c(A)$.\\
(2) Since $\mathbf D(A)$ is a triangulated category with coproducts every idempotent morphism in $\mathbf D(A)$ splits (see \cite[Proposition~1.6.8]{N}). Since $\mathbf D^c(A)=\langle A_A\rangle_{\rm thick}$ is closed under direct summands, also every idempotent in $\mathbf D^c(A)$ splits. Therefore an object $X$ is indecomposable if and only if $0$ and $1$ are the only idempotents in the endomorphism ring $R:=\End_{\D(A)}(X)$ of $X$. This is equivalent to $R$ being indecomposable as module over itself, and this is the same as $R\cong\End_R(R)$ being local, since for a compact object $X$ its endomorphism ring $R$ is finite dimensional by (1).\\
Given a compact object, it is the direct sum of finitely many indecomposables by (1) and since the indecomposables have local endomorphism rings by (2) this decomposition is essentially unique. Hence $\D^c(A)$ is a Krull-Remak-Schmidt category.
\end{proof}
\end{lem}

So for a differential graded algebra $A$ with degreewise finite dimensional cohomology, the compact objects in the derived category which we denote by $\D^c(A)$ form a triangulated $\Hom$-finite $k$-linear Krull-Remak-Schmidt category. The following proposition gives some characterisations for the existence of Auslander-Reiten triangles in this category.


\begin{prop}\label{prop:gor}
Let $A$ be a differential graded algebra with degreewise finite dimensional cohomology. Then the following conditions are equivalent:
\begin{enumerate}
\item The category $\D^c(A)$ has Auslander-Reiten triangles.
\item The category $\D^c(A)$ has Serre duality, i.e.\ there exists an exact autoequivalence $S\colon\D^c(A)\to\D^c(A)$ and natural isomorphisms $$D\Hom_{\D(A)}(X,Y)\cong\Hom_{\D(A)}(Y,SX)\text{ for all }X,Y\in\D^c(A).$$ 
\item $_A (DA)\in\D^c(A^\op)$ and $(DA)_A\in\D^c(A)$.
\item The $k$-duality $D$ restricts to a duality on the compact objects
$$\xymatrix{
D^c(A) \ar@<1ex>[r]^D & D^c(A^\op) \ar@<1ex>[l]^D
}$$


\end{enumerate}
Moreover, in this situation the Serre functor is $S=-\otimes^L_{A}DA$ and the Auslander-Reiten translate is $(-\otimes^L_{A}DA)\circ\Sigma^{-1}$.
\begin{proof}
The equivalence of (1) and (2) is shown in \cite[Proposition~I.2.3]{RV} as well as in \cite[Theorem~4.4]{K1}. It follows from Lemma~\ref{lem:duality} that the Serre functor is $-\otimes^L_{A}DA$. For the equivalence of (2),(3) and (4) one can use d\'evissage arguments.

\end{proof}
\end{prop}

\begin{rem}
Following \cite{K0} or \cite{RV} one can construct the Auslander-Reiten triangle corresponding to an indecomposable compact object $Z$ as follows. Serre duality gives an isomorphism
$$D\End_{\D(A)}(Z)=D\Hom_{\D(A)}(Z,Z)\cong\Hom_{\D(A)}(Z,Z\otimes_A^L DA)$$
Let $\Gamma:=\End_{\D(A)}(Z)$. Since $\Gamma$ is a finite dimensional local algebra, the projective cover of $\Gamma$-modules $\Gamma\to\Gamma/{\rad\Gamma}$ gives rise to an injective envelope of $\Gamma$-modules $\Gamma/{\rad\Gamma}\to D\Gamma$. The corresponding canonical map $\Gamma\to\Gamma/{\rad\Gamma}\to D\Gamma$ is given by an element in $D\Gamma=D\End_{\D(A)}(Z)$.
Take the map $Z\to Z\otimes_A^L DA$ that corresponds to this element under the Serre duality isomorphism and complete it to a triangle. This gives an Auslander-Reiten triangle
$$\Sigma^{-1}Z\otimes_A^L DA\to Y\to Z\to Z\otimes_A^L DA.$$
\end{rem}

\begin{rem}\label{rem:gor}
In the case of a finite dimensional $k$-agebra considered as differential graded $k$-algebra concentrated in degree zero, the derived category $\D(A)$ is the usual (unbounded) derived category $\D(\Mod A)$ of the category of all $A$-modules and $\D^c(A)$ is the category of perfect complexes, i.e.\ complexes isomorphic in $\D(A)$ to a bounded complex of finitely generated projective modules. In this situation we have:
\begin{enumerate} 
\item The third condition in Proposition~\ref{prop:gor} says that $A$ and $A^\op$ have finite injective dimension. This is the definition of a Gorenstein algebra given by Auslander-Reiten in \cite{A}. 
\item In the case where $A$ has finite global dimension, Happel's result in \cite{H}, where he shows that $\D^b(\mod A)$ has Auslander-Reiten triangles, is a special case of the proposition since in this case $\D^b(\mod A)\simeq\D^c(\Mod A)=\D^c(A)$.
\end{enumerate}
\end{rem}


\begin{rem}\label{rem:CY}
Let $A$ be a differential graded algebra with degreewise finite dimensional cohomology. In addition, let $\D^c(A)$ be {\em Calabi-Yau} of dimension $d$, i.e.\ the Serre functor $S=-\otimes^L_ADA$ is naturally isomorphic to $\Sigma^d$, the $d$-fold composition of the shift functor. In this case $\D^c(A)$ has Auslander-Reiten triangles, and we clearly have that $DA\cong\Sigma^dA$ is compact. The Auslander-Reiten translate in this case is $S\circ\Sigma^{-1}\cong\Sigma^{d-1}$.
\end{rem}

\newpage

\section{Simply connected differential graded algebras}

Let $k$ be a field. In the following we will consider {\em simply connected differential graded $k$-algebras of finite type}, i.e.\ positively graded differential graded $k$-algebras $A$ with finite dimensional cohomology algebra, $\Ho^0A\cong k$, and $\Ho^1A=0$. For simplicity, at some stages in Sections~3.3 and 3.4 we will assume that our simply connected differential graded algebras are {\em augmented}, i.e.\ the natural morphism of differential graded algebras $k\to A$ has a splitting $A\to k$ called the augmentation morphism. In particular, $k$ gets the structure of a differential graded $A$-module. All results in these sections carry over to not necessarily augmented simply connected differential graded algebras but $k_A$ has to be replaced by an appropriate object. 

\subsection{A finite dimensional model}
Up to weak equivalence (i.e.\ equivalence via a series of quasi-isomorphisms) one can replace a simply connected differential graded algebra of finite type by a finite dimensional differential graded algebra $A$ with $A^0=k,~A^1=0$, and $\sup A=\sup\Ho^*A=:d$.
$$A\colon~\cdots\to 0\to k\to 0\to A^2\to\cdots\to A^d\to 0\to\cdots$$
Note that $A$ is an augmented differential graded algebra with augmentation morphism $A\to A/A^{\geq 1}\cong k$.

The construction of such a finite dimensional model proceeds in two steps:
\begin{enumerate}
\item Inductive construction of a {\em free model} using graded tensor algebras.
\item Right truncation of the free model (cf.\ \cite[Section~12 Example 6]{FHT}).
\end{enumerate}
ad (1): The construction has its origin in \cite{AHi} in which the so-called Adams-Hilton model for the singular chains of the loop space of a simply connected CW-complex is introduced. The construction of the free model for our simply connected differential graded algebra $A$, that we will recall now, is the same as in the graded commutative case where one can construct a minimal Sullivan model using the free graded commutative algebra (see \cite[12.2 Proposition]{FHT}), but one uses the graded tensor algebra instead of its graded commutative quotient. Given a graded vector space $V$, let $TV:=\coprod_{i=0}^\infty V^{\otimes i}$ be the associated {\em graded tensor algebra}. The following two properties of the tensor algebra are important for the inductive construction of the free model.
\begin{enumerate}
\item[(i)] Every $k$-linear map of degree zero from $V$ to a graded algebra $B$ can uniquely be extended to a morphism of graded algebras $TV\to B$.
\item[(ii)] Every $k$-linear map of degree $n$ from $V$ to $TV$ can be uniquely extended to a derivation $d\colon TV\to TV$ of degree $n$, i.e.\ a $k$-linear map of degree $n$ satisfying $d(xy)=d(x)y+(-1)^{n |x|}xd(y)$ for all $x,y\in TV$ and $x$ homogeneous.
\end{enumerate}
One starts with the vectorspace $V^2=\Ho^2(A)$ in degree 2 and defines $m_2\colon (TV^2,0)\to (A,d)$ by extending the $k$-linear map that assigns representing cocycles to the elements of a chosen basis of $\Ho^2(A)$. Clearly $\Ho^i(m_2)$ is an isomorphism for $i\leq 2$ and $\Ho^3(m_2)$ is injective since $(TV^2)^3=0$. If $m_k\colon (TV^{\leq k},d)\to(A,d)$ has been constructed, then let $V^{k+1}$ be the vectorspace in degree $k+1$ that is the direct sum of a complement $C$ of $\Im\Ho^{k+1}(m_k)$ in $\Ho^{k+1}(A)$ and $\Ker\Ho^{k+2}(m_k)$. Extend the differential on $TV^{\leq k+1}$ by sending $C$ to zero and the elements of a chosen basis $\{\overline{z_\a}\}$ of $\Ker\Ho^{k+2}(m_k)$ to corresponding representing cocycles $z_\a$. Define $m_{k+1}\colon (TV^{\leq k+1},d)\to(A,d)$ by sending the elements of a basis of $C$ to corresponding representing cocycles and the elements $\overline{z_\a}$ of the chosen basis of $\Ker\Ho^{k+2}(m_k)$ to elements in the preimage under $d$ of $m_k(z_a)$.
\\
ad (2): Note that there is a vectorspace decomposition
$$A^i\cong\Ker(d^i)\oplus\Im(d^i)\cong\Im(d^{i-1})\oplus\Ho^iA\oplus\Im(d^i).$$
We choose some $k$-linear splitting $p^i$ for $d^i\colon A^i\to\Im(d^i)$. Using that $A$ is positively graded, $A^0=k$, and $A^1=0$ one sees that the following subcomplex is a differential graded ideal in $A$.
$$\xymatrix@=4mm{\cdots \ar[r] & k \ar[r] & 0 \ar[r] & A^2 \ar[r]^{d^2} & \cdots \ar[r] & A^{d-2} \ar[r]^{d^{d-2}} & A ^{d-1} \ar[r]^{d^{d-1}} & A^{d} \ar[r]^{d^d} & A^{d+1} \ar[r] & \cdots \\
\\
\cdots \ar[r] & 0 \ar[r]\ar@_{{(}->}[uu] & 0 \ar[r]\ar@_{{(}->}[uu] & 0 \ar[r]\ar@_{{(}->}[uu] & \cdots \ar[r] & 0 \ar[r]\ar@_{{(}->}[uu] & \Im(d^{d-1}) \ar[r]^<<<{1\brack 0}\ar@_{{(}->}[uu]^{p^{d-1}} & \Im(d^{d-1})\oplus\Im(d^d)\ar[r]^>>{[0~\text{incl}]}\ar@_{{(}->}[uu]^{[\text{incl}~p^d]} & A^{d+1} \ar[r]\ar@{=}[uu] & \cdots
}$$
The quotient differential graded algebra is then the desired algebra.

\begin{rem}\label{rem:model}
(1) Since quasi-isomorphic differential graded algebras give equivalent derived categories (see Lemma~\ref{lem:equ}) and we are interested in properties of the derived category we may always replace our simply connected differential graded algebra of finite type by a differential graded algebra $A$ that is a finite dimensional positively graded differential graded $k$-algebra with $A^0=k,~A^1=0$, and $\sup A=\sup\Ho^*A$. In the following we will refer to such a replacement as a {\em finite dimensional model}.\\
(2) If we start with an augmented simply connected differential graded algebra $A$ of finite type then the finite dimensional model is weakly equivalent to $A$ in the category of augmented differential graded algebras.
\end{rem}

\begin{exm}
The singular cochain differential graded algebra $C^*(X;k)$ of a simply connected topological space $X$ with finite dimensional cohomology is a simply connected differential graded algebra of finite type.  The Hurewicz theorem gives $\Ho_0(X)\cong k,~\Ho_1(X)=0$ and therefore $\Ho^0(X)\cong k,~\Ho^1(X)=0$. The cohomology algebra $\Ho^*(X;k)$ is graded commutative, and provided the characteristic of $k$ is zero it is even possible to construct a graded commutative finite dimensional model for $C^*(X;k)$ (see \cite{FHT}). If we fix a base point $x$ in the topological space $X$, then the inclusion map $\{x\}\to X$ induces an augmentation morphism $C^*(X;k)\to k$.
\end{exm}

\subsection{Minimal semi-free resolutions}
Let us now assume that our simply connected differential graded algebra $A$ of finite type is finite dimensional, positively graded with $A^0=k,~A^1=0$, and $\sup A=\sup \Ho^*A$.
$$A\colon~\cdots\to 0\to k\to 0\to A^2\to\cdots\to A^d\to 0\to\cdots$$
For all differential graded modules over differential graded algebras of this form 
there exist minimal semi-free resolutions (cf.\ \cite{AFH2}). 
A {\em semi-free resolution} of $M$ is a quasi isomorphism $L\to M$, where $L$ is {\em semi-free}, i.e.\ there exists an ascending exhaustive filtration of $A$-modules
$$0\subseteq L(0)\subseteq L(1)\subseteq\dots\subseteq L(i)\subseteq\dots\subseteq L,\quad\quad\bigcup_{i\geq 0} L(i)=L$$
with subquotients $L(i)/L(i-1)$ that 
are isomorphic in $\mathbf C(A)$ to direct sums of shifts of copies of $A$. 
Note that a semi-free differential graded module is in particular homotopically projective (see Lemma~\ref{lem:homproj}). A differential graded module $L$ is called {\em minimal} if every morphism of differential graded modules $L\to L$ that is a homotopy equivalence is already an isomorphism. In the situation of our particular differential graded algebra $A$ an $A$-module $L$ is minimal if and only if $$d(L)\subseteq\rad A\cdot L=A^{\geq 1}L.$$
This implies: If $L$ is minimal, then $\H om_A(L,k)$ and $L\otimes_Ak$ have vanishing differentials. The existence of minimal semi-free resolutions for differential graded modules $M$ with $\inf\Ho^*M>-\infty$ was stated without a proof in \cite{FHT2}. In particular this gives the existence of minimal semi-free resolutions for differential graded modules that are compact in the derived category. In the following form this result is stated and proved in \cite{FJ2}.

%

\begin{lem}\label{lem:sfr}
Let $A$ be a differential graded algebra that is finite dimensional, positively graded with $A^0=k,~A^1=0$, and $\sup A=\sup \Ho^*A$. Further let $M$ be a differential graded $A$-module with $\inf\Ho^*M=i>-\infty$§ and $\Ho^n(M)$ finite dimensional for all $n\in\mathbb Z$. Then the following hold.
\begin{enumerate}
\item There exists a minimal semi-free resolution $L\to M$ which has a semi-free filtration with quotients as indicated,
$$\xymatrix@!0{
& \Sigma^{-i}A^{(\gamma_0)} &&&& \Sigma^{-i}A^{(\gamma_1)} &&&& \Sigma^{-(i+1)}A^{(\gamma_2)} &&& \cdots \\
0\ar@{-}[ur] & \subseteq & F(0) \ar@{-}[ul]\ar@{-}[dr] & \subseteq & G(1) \ar@{-}[ur]\ar@{-}[dl] & \subseteq & F(1) \ar@{-}[ul]\ar@{-}[dr] & \subseteq & G(2) \ar@{-}[ur]\ar@{-}[dl] & \subseteq & F(2) \ar@{-}[ul] & \subseteq & \cdots & \subseteq & L \\
&&& \Sigma^{-(i+1)}A^{(\delta_0)} &&&& \Sigma^{-(i+2)}A^{(\delta_1)} &&&&& \cdots
}$$
with $\gamma_j$ and $\delta_j$ finite.
\item $L^\natural\cong\coprod_{j\geq i}\Sigma^{-j}(A^\natural)^{(\beta_j)}$, where $\beta_i=\gamma_0+\gamma_1$ and $\beta_j=\delta_{j-(i+1)}+\gamma_{j-i+1}$ for $j>i$.
\item If the filtration terminates, then there exists a {\em semi-split} exact sequence (i.e.\ it splits as a sequence of the underlying graded modules)
$$0\to P\to L\to\Sigma^{-w}A^{(\a)}\to 0$$
with $\a\neq 0$ finite, $\e_j$ finite, $w\geq i$, P being homotopically projective, and
$$P^\natural\cong\coprod_{j\leq w}\Sigma^{-j}(A^\natural)^{(\epsilon_j)}.$$
\end{enumerate}
\end{lem}

Following the proof of this lemma in \cite{FJ2} we describe the construction of such a minimal semi-free resolution: One starts with a morphism 
$$\a\colon F(0):=\Sigma^{-i}A^{(\gamma_0)}\to M$$ such that $$\Ho^i(\a)\colon\Ho^i(\Sigma^{-i}A^{(\gamma_0)})\to\Ho^i(M)$$
 is an isomorphism. If one has constructed 
 $$\a(n-1)\colon F(n-1)\to M$$ such that
 $$\Ho^{i+j}(\a(n-1))\colon\Ho^{i+j}F(n-1)\to\Ho^{i+j}M$$ is an isomorphism for $0\leq j\leq n-1$, then one first extends $\a(n-1)$ to 
 $$\beta(n)\colon {G(n):=\Sigma^{-(i+n)}A^{(\delta_n)}\amalg F(n-1)}\to M$$ such that $$\Ho^{i+n}(\beta(n))\colon\Ho^{i+n}G(n)\to\Ho^{i+n}M$$ is surjective, $\Ho^{i+n}(\beta(n))(\Ho^{i+n}(\Sigma^{-(i+n)}A^{(\delta_n)}))$ is complemental to $\Ho^{i+n}(\beta(n))(\Ho^{i+n}(F(n-1))$, and $\Ker(\Ho^{i+n}(\beta(n))\subseteq\Ho^{i+n}(F(n-1))$. In the next step one `kills' representing cocycles in $\Ker(\Ho^{i+n}(\beta(n))$ by constructing the mapping cone of a morphism
 $$\delta\colon\Sigma^{-(i+n)}A^{(Y)}\to G(n),~1_y\mapsto y$$ with $\Ho^{i+n}(\delta)$ injective, and $Y$ a set of representing cocycles such that their cohomology classes generate $\Ker(\Ho^{i+n}(\beta(n))$. One extends $\beta(n)$ to
 $$\a(n)\colon F(n)\to M$$ by mapping $1_y\in\Sigma^{-(i+n)+1}A^{(Y)}$ to $m$, where $\beta(n)(y)=d(m)$.

\begin{exm}
A minimal semi-free resolution of $k_A$ can be constructed as follows: First one considers the projection $m_0\colon A\to k_A$. Then degree by degree one inductively `kills' cocycles in a basis of $\Ker\Ho^*(m_i)$.
\end{exm}

\subsection{J\o rgensen's results}
Using minimal semi-free resolutions J\o rgensen shows the following characterisations for the existence of Auslander-Reiten triangles in $\D^c(A)$.

\begin{thm}\cite[Theorem~5.1]{J1}\label{thm:poincare}
Let $A$ be a simply connected differential graded $k$-algebra of finite type and $d:=\sup\Ho^*A$. Then the following conditions are equivalent.
\begin{enumerate}
\item The category $\D^c(A)$ has Auslander-Reiten triangles.
\item The category $\D^c(A)$ is {\em Calabi-Yau} of dimension $d$, i.e.\ the Serre functor is naturally isomorphic to $\Sigma^d$.
\item The differential graded algebra $A$ satisfies {\em Poincar\'e duality} of dimension $d$, i.e.\ $D\Ho^*A$ and $\Sigma^d\Ho^*A$ are isomorphic as left $\Ho^*A$-modules and as right $\Ho^*A$-modules.
\item The objects $DA$ and $\Sigma^dA$ are isomorphic in $\D(A)$ and in $\D(A^\op)$.
\end{enumerate}
If $A$ is augmented, then the conditions are also equivalent to:
\begin{enumerate}
\item[(5)] The differential graded algebra $A$ is {\em Gorenstein} of dimension $d$ in the sense of Avramov-Foxby \cite{AF} respectively F\'elix-Halperin-Thomas \cite{FHT2}, i.e.\ there are isomorphisms of graded $k$-vectorspaces
$$\Ext^*_A(k,A)\cong\Sigma^{-d}k\text{ and } \Ext^*_{A^\op}(k,A)\cong\Sigma^{-d}k.$$
\end{enumerate}
\end{thm}

In the following we will refer to a simply connected differential graded algebra $A$ of finite type satisfying the equivalent conditions in Theorem~\ref{thm:poincare} as simply connected {\em Gorenstein} differential graded algebra of finite type. Note that we do so even if $A$ is not augmented. If we use the notion Gorenstein for a differential graded algebra in a different context, then this will be mentioned explicitly.

\begin{rem}\label{rem:gorfj}
In \cite{FJ} Frankild-J\o rgensen define a differential graded algebra to be Gorenstein if the contravariant functor $$\RHom_A(-,A)\colon\D^{f,b}(A)\to\D^{f,b}(A^\op)$$ is a duality. Note that $\D^{f,b}(A)$ consists of those differential graded $A$-modules $X$ with $\dim\Ho^*X<\infty$. Let the characteristic of the field $k$ be zero and $A$ be an augmented simply connected differential graded algebra of finite type that is graded commutative in cohomology. Then it is shown in \cite{FIJ} that the Gorenstein condition of F\'elix-Halperin-Thomas \cite{FHT2} respectively Avramov-Foxby \cite{AF} (condition (5) in the theorem) is equivalent to the definition of Gorensteinness of Frankild-J\o rgensen. We also recommend \cite{DGI} to the reader for another consideration of Gorenstein differential graded algebras.
\end{rem}

\begin{rem}\label{rem:top}
For $A=C^*(X;k)$ the singular cochain differential graded algebra of a simply connected topological space $X$ the condition (3) in Theorem~\ref{thm:poincare} says that $X$ has Poincar\'e duality.
\end{rem}

%

In addition to the result about the existence of Auslander-Reiten triangles, J\o rgensen also studied the shape of the components that appear in the corresponding Auslander-Reiten quiver $\Gamma(A)$. He formulated and proved the following result under some additional assumptions that we have been able to remove (see the discussion below and Section~3.4).

\begin{thm}\cite[Theorem~4.2]{J2}\label{thm:ZA} Let $A$ be a simply connected Gorenstein differential graded algebra of finite type with $\dim\Ho^*A\geq 2$. Then all connected components of the Auslander-Reiten quiver of $\D^c(A)$ are of the form $\mathbb Z A_\infty$.
\end{thm}
For the proof of the theorem J\o rgensen replaces $A$ by a finite dimensional model that we also denote by $A$ and considers the function
$$f\colon\Ob(\D(A))\to\mathbb N_0\cup\{\infty\},\quad M\mapsto\dim\Ho^*\RHom_A(M,k).$$
He shows that $f$ is an additive, unbounded, Auslander-Reiten translate periodic function on each component of the Auslander-Reiten quiver of $\D^c(A)$. Therefor he needs that $k_A$ is not compact and shows that this is the case for the singular cochain differential graded algebra $C^*(X;k)$ of a simply connected Poincar\'e duality space $X$ of Poincar\'e dimension $d\geq 2$ provided the characteristic of $k$ is zero. In fact, this is true for augmented simply connected differential graded algebras $A$ of finite type with $\dim\Ho^*A\geq 2$ in general as will follow from the amplitude inequality (Proposition~\ref{prop:ampl}) in Section~3.4.

We remark that the function $f$ has the following properties.

\begin{lem}
\begin{enumerate}
\item The function $f$ is subadditive on triangles, i.e.\ for all exact triangles $X\to Y\to Z\to\Sigma X$ it holds $f(Y)\leq f(X)+f(Z)$.
\item An object $M\in\D(A)$ is compact if and only if $\dim\Ho^*M<\infty$ and $f(M)<\infty$. Moreover, in this case $f(M)=\dim\Ho^*\RHom_A(M,k)=\dim\Ho^*(M\otimes^L_A k)$ counts the number of shifted copies of $A$ that appear in the minimal semi-free resolution of $M$.
\end{enumerate}
\end{lem}
\begin{proof}
(1) This follows immediately from the fact that $\RHom_A(-,k)\colon\D(A)\to\D(A)$ is an exact and $\Ho^*\colon\D(A)\to\mathbf C_{gr}(k)$ a cohomological functor and looking at the corresponding long exact sequence.\\
(2) If $M$ is compact, then it is finitely built from $A$, hence $\dim\Ho^*M<\infty$ and $f(M)<\infty$. On the other hand if $\dim\Ho^*M<\infty$ and $f(M)<\infty$, then there is a minimal semi-free resolution $L\to M$ as in Lemma~\ref{lem:sfr} with semi-free filtration that terminates after finitely many steps. It follows that $M$ is finitely built from $A$, hence $M$ is compact. Now let $M$ be compact and $L\to M$ be a minimal semi-free resolution as in Lemma~\ref{lem:sfr}. We have
\begin{multline*}
\RHom_A(M,k)=\H om_A(L,k)\cong\H om_A(L,k)^\natural\cong\Hom_{A^\natural}(\coprod_{j\leq-u}\Sigma^j(A^\natural)^{(\beta_j)},k^\natural)\\
\cong\prod_{j\leq-u}\Sigma^j k
\end{multline*}
and
$$M\otimes^L_A k=L\otimes_A k\cong(L\otimes_A k)^\natural\cong(\coprod_{j\leq-u}\Sigma^j(A^\natural)^{(\beta_j)}\otimes_A k^\natural)\cong\coprod_{j\leq-u}\Sigma^j k.$$
So if $M$ is compact, then $f(M)=\dim\Ho^*\RHom_A(M,k)=\dim\Ho^*(M\otimes^L_A k)$ counts the number of shifted copies of $A$ that appear in the minimal semi-free resolution $L$ of $M$.\\
\end{proof}

For spheres J\o rgensen computes the Auslander-Reiten quiver. We remark that he has stated the following result in the case where the ground field $k$ has characteristic zero, but the proof goes through also in arbitrary characteristic.

\begin{exm}\cite[Theorem~8.13]{J1}\label{exm:sphere}
For the $d$-dimensional sphere $S^d$, $d\geq 2$ the Auslan-der-Reiten quiver $\Gamma(S^d)$ consists of $d-1$ components of the form $\mathbb Z A_\infty$. 
\end{exm}

We will come back to this example in Section~6. For this reason we will sketch J\o rgensen's arguments:
The singular cochain differential graded algebra $C^*(S^d;k)$ of a $d$-dimensional sphere $S^d$ is weakly equivalent to its cohomology differential graded algebra $\Ho^*(S^d)\cong k[x]/(x^2)=:A,$ where $\deg x=d$. Note that $A$ is a finite dimensional model for $C^*(S^d;k)$ in the sense of Remark~\ref{rem:model}(1). Consider $B=k[y]$ with $\deg y=-d+1$. The mapping cone $L$ of the canonical morphisms $$\Sigma^{d-1}B\to B$$ is a minimal semi-free resolution of $k_B$. Let $\E:=\mathcal{E}nd_B(L)$ be the endomorphism differential graded algebra of $L$.
Consider the adjoint pair of functors 
$$\xymatrix{
  \D(\E)\ar@<1ex>[rr]^{-\otimes_\E^Lk} && \D(B)\ar@<1ex>[ll]^{\RHom_{B}(k,-)}
  }.$$
In \cite{DG} it is shown that these functors restrict to an equivalence $\D(\E)\simeq\langle k_B\rangle_{loc}\subseteq\D(B)$. Note that $\langle k_B\rangle_{loc}\subseteq\D(B)$ is closed under direct summands and that $\E$ and $A$ are quasi-isomorphic differential graded algebras. Hence we have an equivalence $$\D^c(S^d)\simeq\D^c(\Ho^*(S^d;k))\simeq\D^c(A)\simeq\D^c(\E)\simeq\langle k_B\rangle_{\rm thick}\subseteq\D(B).$$
J\o rgensen computes the Auslander-Reiten components by analysing $\langle k_B\rangle_{thick}=\D^f(B)$.
It is not clear to us how to produce equivalences like this for examples other than the spheres.

The question of the number of components for a simply connected Gorenstein differential graded algebras of finite type in general will be treated in Section~4.

\subsection{An amplitude inequality}
We give a cochain analogue of the amplitute inequalities in \cite{J3}. Corollary~\ref{cor:ampl} is then exactly the statement that lets us weaken the assumptions on the finite dimensional model of the simply connected Gorenstein differential graded algebra of finite type in Theorem~\ref{thm:ZA} from $k_A$ is not compact to $\dim\Ho^*A\geq 2$.

The {\em amplitude} of a differential graded module $M$ is defined to be 
$$\amp M:=\sup\Ho^*M - \inf\Ho^*M.$$

\begin{prop}\label{prop:ampl}
Let $A$ be a simply connected differential graded algebra of finite type, $0\neq M\in\D^c(A)$, and $0\neq X\in\D(A^\op)$ with $\dim\Ho^*X<\infty$ .
Then
\begin{enumerate}
\item $\inf\Ho^*(M\otimes_A^L X)=\inf\Ho^*M+\inf\Ho^*X$.
\item $\sup\Ho^*(M\otimes_A^L X)\geq\inf\Ho^*M+\sup\Ho^*X$.
\item $\amp(M\otimes_A^L X)\geq\amp X$.
\end{enumerate}
In particular, $\amp M\geq\amp A$.
\end{prop}
\begin{proof}
By Lemma~\ref{lem:equ}(2) we can replace $A$ with a finite dimensional model. Furthermore, without loss of generality we may assume that $X\in\D(A^\op)$ satisfies $\inf X=\inf\Ho^*X$ and $\sup X=\sup\Ho^*X$. This follows from the construction of a minimal semi-free resolution of $X$ as described in Lemma~\ref{lem:sfr} and right truncation as done with the free model of the simply connected differential graded algebra of finite type in the beginning of this section. Since $M$ is compact there exists a semi-free resolution $L\to M$ as in Lemma~\ref{lem:sfr} with semi-free filtration of $L$ that terminates after finitely many steps.\\
(1) Of course we have
$$\inf\Ho^*(M\otimes_A^L X)\geq\inf (M\otimes_A^L X)=\inf (L\otimes_A X)=\inf L+\inf X=\inf\Ho^*M+\inf\Ho^*X.$$
On the other hand two non zero  elements
$$\bar m\in\Ho^{\inf\Ho^*M}L\quad\text{and}\quad\bar x\in\Ho^{\inf\Ho^*X}X$$
give a non zero element 
$$\overline{m\otimes_A x}\in\Ho^{\inf\Ho^*M+\inf\Ho^*X}(L\otimes_A X)=\Ho^{\inf\Ho^*M+\inf\Ho^*X}(M\otimes_A^L X).$$
Hence it also holds $\inf\Ho^*(M\otimes_A^L X)\leq\inf\Ho^*M+\inf\Ho^*X$.
\\
(2) We take from Lemma~\ref{lem:sfr}~(iii) the semi-split short exact sequence
$$0\to P\to L\to\Sigma^{-w}A^{(\a)}\to 0$$
with $\a\neq 0$ finite, $w\geq\inf\Ho^*M$, and $P^\natural\cong\coprod_{j\leq w}\Sigma^{-j}(A^\natural)^{(\epsilon_j)}$.
Then we apply $-\otimes_A X$ and get again a semi-split short exact sequence
$$0\to P\otimes_A X\to L\otimes_A X\to \Sigma^{-w}A^{(\a)}\otimes_A X\to 0.$$
This gives a long exact cohomology sequence
$$\cdots\to\Ho^{w+\sup X}(L\otimes_A X)\to\Ho^{w+\sup X}(\Sigma^{-w} A^{(\a)}\otimes_A X)\to\Ho^{w+\sup X+1}(P\otimes_A X)\to\cdots$$
We have
$(P\otimes_A X)^\natural\cong\coprod_{j\leq w}\Sigma^{-j}(A^\natural)^{(\epsilon_j)}\otimes_A X^\natural\cong\coprod_{j\leq w}\Sigma^{-j} (X^\natural)^{(\epsilon_j)}$.
Hence 
$$\Ho^{w+\sup\Ho^*X+1}(P\otimes_A X)=(P\otimes_A X)^{w+\sup\Ho^*X+1}=0.$$
It follows that $\Ho^{w+\sup\Ho^*X}(L\otimes_A X)$ maps onto
$$\Ho^{w+\sup\Ho^*X}(\Sigma^{-w} A^{(\a)}\otimes_A X)\cong \Ho^{\sup\Ho^*X}(X^{(\a)})\neq 0.$$
Therefore, $\Ho^{w+\sup\Ho^*X}(M\otimes_A^L X)=\Ho^{w+\sup\Ho^*X}(L\otimes_A X)\neq 0$ and we finally get $$\sup\Ho^*(M\otimes_A^L X)\geq w+\sup\Ho^*X\geq\inf\Ho^*M+\sup\Ho^*X.$$
(3) This follows directly by combining (1) and (2).
\end{proof}

The following corollary is now an immediate consequence of the previous proposition.

\begin{cor}\label{cor:ampl}
Let $A$ be an augmented simply connected differential graded algebra of finite type with $\dim\Ho^*A\geq 2$. Then $k_A\in\D(A)$ is not compact.
\end{cor}

To prove his result about the shape of the Auslander-Reiten components J\o rgensen needs that $k_A$ is not compact if $\dim\Ho^*A\geq 2$. 
By our result this assertion is now generally fulfilled for augmented simply connected differential graded algebra of finite type and in particular independent of the characteristic of the ground field.

\newpage

\section{The number of Auslander-Reiten components}

Let $A$ be a simply connected Gorenstein differential graded algebra of finite type. By Theorem~\ref{thm:poincare} the category $\D^c(A)$ has Auslander-Reiten triangles. The following theorem states that the situation in Example~\ref{exm:sphere}, namely that the Auslander-Reiten quiver of a sphere consists of Gorenstein-dimension minus $1$ components is quite special. To follow the proof of the theorem we suggest the reader for the first reading to look at some example (cf.\ Example~\ref{exm:S2S2}).

\begin{thm}\label{thm:noc} Let $A$ be a simply connected Gorenstein differential graded algebra of finite type. Then the following hold.
\begin{enumerate}
\item The Auslander-Reiten quiver of $A$ has finitely many components if and only if $\dim \Ho^*A=2$. In this case the number of components equals $\sup\{i\mid\Ho^iA\neq 0\}-1$.
\item If $\dim \Ho^eA\geq 2$ for some $e$, then there is an $n$-parameter family of Auslander-Reiten components for each $n\in\mathbb N$. In fact there are objects, each lying in a different component, that can be parametrised by $\mathbb P^1(k)^n$.
\end{enumerate}
\end{thm}
\begin{proof}
Let $d:=\sup\Ho^*A$. If $\dim\Ho^*A=1$, then $A$ is quasi-isomorphic to the ground field $k$ viewed as differential graded algebra concentrated in degree zero, and the Auslander-Reiten quiver clearly consists of countably many components each of them containing just a single object. If $\dim\Ho^*A=2$, then $A$ is quasi-isomorphic to the differential graded algebra $k[x]/{(x^2)}$ with $\deg x=d$. These are the differential graded algebras weakly equivalent to $C^*(S^d;k)$ for some $d\geq 2$, for those J\o rgensen has computed that the Auslander-Reiten quiver consists of $d-1$ components of the form $\mathbb ZA_\infty$ (see Example~\ref{exm:sphere}). If $\dim \Ho^*A\geq 3$, then we will get indecomposable modules that lie in different components via iterated mapping cone constructions. Under the additional assumption in (2) we are in this way even able to construct $n$-parameter families of modules in different components. Before we continue with the proof, we will make some preliminary observations and explain the skeletal structure of the mapping cone constructions.

\subsection{Preparations for the proof}

In the following we will always assume that $\dim \Ho^*A\geq 2$. As Theorem~\ref{thm:ZA} tells us, the components of the Auslander-Reiten quiver of $A$ are of the form $\mathbb Z A_\infty$ and the Auslander-Reiten translate is $\t=\Sigma^{d-1}$ hence the components look as follows:
$$
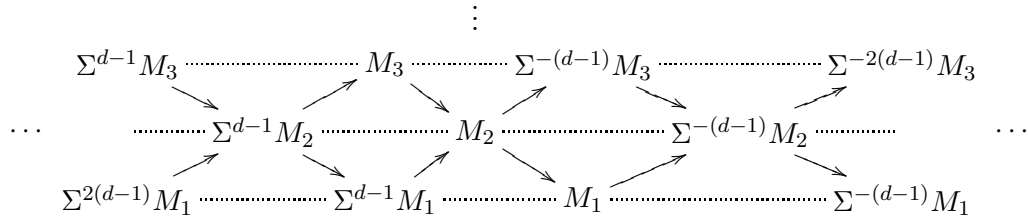
\xymatrix@=1pt{
&&&& \vdots \\
&\Sigma^{d-1}M_3 \ar[rdd] \ar@{.}[rr] & & M_3 \ar[rdd] \ar@{.}[rr] & & \Sigma^{-(d-1)}M_3 \ar[rdd] \ar@{.}[rr] & & \Sigma^{-2(d-1)}M_3 \\
\\
\cdots & \ar@{.}[r] & \Sigma^{d-1}M_2 \ar[rdd] \ar[ruu] \ar@{.}[rr] & & M_2 \ar[rdd] \ar[ruu] \ar@{.}[rr] & & \Sigma^{-(d-1)}M_2 \ar[rdd] \ar[ruu] \ar@{.}[r] & & \cdots \\
\\
&\Sigma^{2(d-1)}M_1 \ar[ruu] \ar@{.}[rr] &  & \Sigma^{d-1}M_1 \ar[ruu] \ar@{.}[rr] & & M_1 \ar[ruu] \ar@{.}[rr] & & \Sigma^{-(d-1)}M_1 \\
}$$
Recall from Section~3.3 the function $f:=\dim\Ho^*\RHom_A(-,k)$. Since $f$ is additive and $f\circ\Sigma=f$, it holds $f(\Sigma^{n(d-1)} M_i)=i\cdot f(M_1)$ for $i\in\mathbb N,~n\in\mathbb Z$. Hence objects in the same component with the same value under $f$ are just iterated $(d-1)$-shifts of each other. If we can construct indecomposable non-isomorphic differential graded $A$-modules that have the same value under $f$ and are not iterated $(d-1)$-shifts of each other, then they have to lie in different components of the Auslander-Reiten quiver. In particular, the number of components of the Auslander-Reiten quiver of $\D^c(A)$ is at least $d-1$ since for an indecomposable compact differential graded $A$-module $M$ all the objects $\Sigma^i M,~i\in\{0,1,2,\dots,d-2\}$ have to lie in different components.

\subsubsection{The basic construction}
We start to describe the skeletal structure of the mapping cone constructions and state some properties of the objects that we get in this way. We will specialise to the specific constructions later.
We start with $C_0:=A$ and some $e_1\in\mathbb N$ with $\Ho^{e_1} A\neq 0$. We define $\p\colon\Sigma^{-e_1}A\to A,~1\mapsto\zeta$ where $\zeta\in A^{e_1}$ with $\bar\zeta\neq 0\in\Ho^{e_1}A$. Let $C_1=C_0(e_1)=A(e_1)$ be by definition the mapping-cone of $\p$.
$$\Sigma^{-e_1}A\overset \p\to A\to C_1\to\Sigma^{-e_1+1}A$$
If there exists some $e_2\in\mathbb N$ with $\Ho^{e_2}C_1\neq 0$, then we could continue the mapping-cone construction starting with $C_1$ and so on, i.e.\
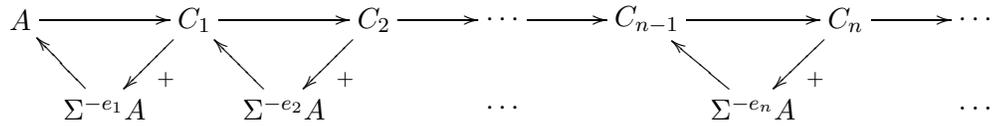
$$\xymatrix@=6pt{
A \ar[rr] && C_1 \ar[rr]\ar[ldd]^+ && C_2 \ar[rrr]\ar[ldd]^+ &&& \cdots \ar[rrr] &&& C_{n-1} \ar[rr] && C_n \ar[ldd]^+\ar[rrr] &&&\cdots \\
\\
& \Sigma^{-e_1}A \ar[luu] && \Sigma^{-e_2}A \ar[luu] &&&& \cdots &&&& \Sigma^{-e_{n}}A\ar[luu] &&&& \dots
}$$
where we denote by $C_n=C_{n-1}(e_{n})$ an object that occurs in an n-step mapping cone construction and by $e_{n}$ the number of shifts of $A$ that belong to the construction in the n-th step.
\begin{rem}
For all $n\in\mathbb N_0$ we have:
\begin{enumerate}
\item[(a)] $C_n$ is compact.
\item[(b)] $C_n$ is minimal semi-free.
\item[(c)] $C_n^\natural\cong A^\natural\oplus\overset n{\underset{i=1}\bigoplus}\Sigma^{-e_i+1}A^\natural$. In particular, a morphism of differential graded $A$-modules starting in $C_n$ is uniquely determined by the images of $1_A$ and $1_{\Sigma^{-e_i+1}A}$, for $i=1,\dots,n$.
\item[(d)] $f(C_n)=n+1.$
\end{enumerate}
\end{rem}
\begin{proof}
(a) This is clear since the compact objects form a triangulated subcategory and we have started with the compact objects $A$ and $\Sigma^{-e_1}A$ and in every construction step $i=1,\dots,n$ we just take the mapping cone of a morphism between the compact objects $\Sigma^{-e_i}A$ and $C_{i-1}$.\\
(b) We have the filtration $A=C_0\subset C_1\subset\dots\subset C_{n-1}\subset C_n$ with quotients $C_i/C_{i-1}\cong \Sigma^{-e_i+1}A,~i=1,\dots,n$. Hence $C_n$ is semi-free. By construction $d(C_n)\subseteq A^{\geq 1}C_n$, i.e.\ $C_n$ is minimal.\\
(d) We calculate
\begin{multline*}
f(C_n)=\dim\Ho^*\RHom_A(C_n,k)=\dim\H om_{A^\natural}(C_n,k)\\
=\dim\H om_{A^\natural}(A^\natural\oplus\bigoplus_{i=1}^n\Sigma^{-e_i+1}A^\natural,k^\natural)=\dim(k\oplus\bigoplus_{i=1}^n\Sigma^{e_i-1}k)=n+1.
\end{multline*}
\end{proof}

We will also need the following
\begin{lem}\label{lem:Ho}
\begin{enumerate}
\item[(i)] $\Ho^jC_n\cong\Ho^jC_{n-1}$ for all $j<e_n$.
\item[(ii)] $\dim\Ho^{e_n}C_n=\dim\Ho^{e_n}C_{n-1}-1$.
\end{enumerate}
\end{lem}
\begin{proof}
(i) We consider the long exact cohomology sequence
$$\xymatrix@=5mm{\cdots \ar[r] & \Ho^{-e_n+j}A \ar[r] & \Ho^{j}C_{n-1} \ar[r] & \Ho^{j}C_n \ar[r] & \Ho^{-e_n+j+1}A \ar[r] & \Ho^{j+1}C_{n-1} \ar[r] & \cdots
}.$$
We clearly have $\Ho^{-e_n+j}A=0$. For $j<e_n-1$ it also holds $\Ho^{-e_n+j+1}A=0$, and hence we get $\Ho^jC_n\cong\Ho^jC_{n-1}$. In the case $j=e_n-1$ the long exact sequence looks as follows.
$$\xymatrix@=5mm{\cdots \ar[r] & \Ho^{-1}A \ar[r]\ar@{=}[d] & \Ho^{e_n-1}C_{n-1} \ar[r] & \Ho^{e_n-1}C_n \ar[r] & \Ho^{0}A \ar@{^{(}->}[r]^<<<<{\neq 0}\ar@{=}[d] & \Ho^{e_n}C_{n-1} \ar[r] & \cdots \\
 & 0 & &  & k
}$$
Since the map $\Ho^0A\to\Ho^{e_n}C_{n-1}$ is non-zero it has to be injective, and so we also get $\Ho^{e_n-1}C_n\cong\Ho^{e_n-1}C_{n-1}$.

(ii) Again we consider the long exact cohomology sequence
$$\xymatrix@=5mm{\cdots \ar[rr]\ar[rd] && \Ho^{0}A \ar@{^{(}->}[r]^<<<<{\neq 0}\ar@{=}[d] & \Ho^{e_n}C_{n-1} \ar[r] & \Ho^{e_n}C_{n} \ar[r] & \Ho^{1}A \ar[r]\ar@{=}[d] & \cdots \\
& 0 \ar[ru] & k & & & 0 &
}.$$
Hence we get $\dim\Ho^{e_n}C_n=\dim\Ho^{e_n}C_{n-1}-1$.
\end{proof}

%

\subsubsection{Indecomposability}
To get indecomposable objects we have to do the constructions more carefully: 

From now on we assume that the differential graded algebra has been replaced by a finite dimensional model that we also denote by $A$. The isomorphism $\Sigma^{d}A\to DA$ in $\D(A)$ gives $\Ho^{d-1}A=0$ and $\Ho^dA\cong k$, where  $d=\sup A$. We will choose the $e_i$ such that $\sup C_{i-1}-d+2\leq e_i\leq\sup C_{i-1}$. Since $\sup C_{i-1}=e_{i-1}+d-1$ this implies $e_{i-1}<e_{i}$ for all $i$.

\begin{rem}\label{rem:Ho1}For all $n\in\mathbb N_0$ we have:
\begin{enumerate}
\item[(i)] $C_n^{\geq\sup C_{n-1}+1}
=\Sigma^{-e_n+1}(A^{\geq \sup C_{n-1}-e_n+2})$
\item[(ii)]$\Ho^{\sup C_n}C_n
\cong k$
\end{enumerate}
\end{rem}

We will call the construction of $C_n=C_{n-1}(e_n)$ with $e_n=\sup C_{n-1}$ a {\em construction step of the first kind}. If $\dim\Ho^*A\geq 3$, then there exists $2\leq e\leq d-2$ having the property $\Ho^eA\neq 0$.
For the first construction step we want to consider the two possibilities
$$e_1=\sup C_0=d~\text{  and }~e_1=e.$$
In the $n$-th construction step we always want to consider the possibility that $e_n=\sup C_{n-1}$, and in case that the $(n-1)$-th construction step has been a construction step of the first kind (i.e.\ $e_{n-1}=\sup C_{n-2}$) also the possibility $e_n=\sup C_{n-1}-d+e$. Here we call the construction $C_n=C_{n-1}(e_n)$ with $e_n=\sup C_{n-1}-d+e$ a {\em construction step of the second kind}. So we allow construction steps of the first and second kind, but after a construction step of the second kind, in the next step we will just allow a construction step of the first kind. This ensures that $e_{n}>\sup C_{n-2}$ and with Remark~\ref{rem:Ho1}(i) it follows $C_{n-1}^{\sup C_{n-1}-d+e}=A^e$ and $\Ho^{\sup C_{n-1}-d+e}C_{n-1}=\Ho^eA\neq 0$. So we will always be able to proceed with our construction. If we define $$e_n^A:=e_n-(\sup C_{n-1}-d)=e_n-e_{n-1}+1,$$
then we have $C_{n-1}^{e_n}=A^{e_n^A}$, and further $e_n^A=d$ if the $n$-th construction step is a step of the first kind and $e_n^A=e$ if the $n$-th construction step is a step of the second kind.

Before we can prove the indecomposability of the objects $C_n$ that are constructed in the described manner, we make the following preliminary observations:
\begin{multline*}\tag{1}
C_n^{e_i}=C_{i+1}^{e_i}=((\Sigma^{-e_{i+1}+1}A)\oplus(\Sigma^{-e_i+1}A)\oplus(\Sigma^{-e_{i-1}+1}A))^{e_i}\\
=A^{e_i-e_{i+1}+1}\oplus A^1\oplus A^{e_i-e_{i-1}+1}=A^{-e_{i+1}^A+2}\oplus A^{e_i^A}
\end{multline*}
\vspace{-1cm}
\begin{multline*}\tag{2}
C_n^{e_i-1}=C_{i}^{e_i-1}=((\Sigma^{-e_i+1}A)\oplus(\Sigma^{-e_{i-1}+1}A)\oplus(\Sigma^{-e_{i-2}+1}A))^{e_i-1}\\
=A^0\oplus A^{e_i-e_{i-1}}\oplus A^{e_i-e_{i-2}}=k\oplus A^{e_i^A-1}\oplus A^{e_i^A+e_{i-1}^A-2}=k\oplus A^{e_i^A-1}\oplus A^{d+e-2}
\end{multline*}
\vspace{-1cm}
\begin{multline*}\tag{3}
C_n^{e_i-2}=C_i^{e_i-2}=((\Sigma^{-e_{i-1}+1}A)\oplus(\Sigma^{-e_{i-2}+1}A)\oplus(\Sigma^{-e_{i-3}+1}A))^{e_i-2}\\
=A^{e_i-e_{i-1}-1}\oplus A^{e_i-e_{i-2}-1}\oplus A^{e_i-e_{i-3}-1}=A^{e_i^A-2}\oplus A^{e_i^A+e_{i-1}^A-3}\oplus A^{e_i^A+e_{i-1}^A+e_{i-2}^A-4}\\
=A^{e_i^A-2}\oplus A^{d+e-3}\oplus A^{d+e+e_i^A-4}
\end{multline*}
\vspace{-1cm}
\begin{multline*}\tag{4}
{C_{n-1}^{e_n-1}=(\Sigma^{-e_{n-1}+1}A)^{e_n-1}=A^{e_n-e_{n-1}}=A^{e_n^A-1}}
\end{multline*}
Let $m\leq n$. As already mentioned before, a morphism of differential graded modules $a\colon C_m\to C_n$ is determined by the images of $1_A$ and $1_{\Sigma^{-e_i+1}A},~i=1,\dots,m$. These images lie in $A^0=k$ and $C_n^{e_i-1}=k\oplus A^{e_i^A-1}\oplus A^{d+e-2}$ respectively. Say $ 1_A\mapsto\a_{0}$ and $1_{\Sigma^{-e_i+1}A}\mapsto\a_i+x_i+y_i$ for $i=1,\dots,m$.
Since the differential $d_{C_n}^{e_i-1}\colon C_n^{e_i-1}\to C_n^{e_i}$ is of the form $[\p,(-1)^{e_{i-1}-1}d_A^{e_i^A-1},0]\colon k\oplus A^{e_i^A-1}\oplus A^{d+e-2}\to A^{e_i^A}$, we have
$$(d\circ a)(1_{\Sigma^{-e_n+1}A})=d(\a_i+x_i+y_i)=\p(\a_i)+(-1)^{e_{i-1}-1}d_A(x_i)=\zeta\a_i+(-1)^{e_{i-1}-1}d_A(x_i),$$
$$(a\circ d)(1_{\Sigma^{-e_n+1}A})=a(\zeta)=a(1_{\Sigma^{-e_{i-1}+1}A})\cdot\zeta=\a_{i-1}\zeta=\zeta\a_{i-1},$$
and we get $\zeta(\a_i-\a_{i-1})=(-1)^{e_{i-1}-1}d_A(x_i)$ for $i=1,\dots,m$. Since $\zeta$ is a not a boundary, it follows $\a_i=\a_{i-1}$ and $d(x_i)=d_A(x_i)=0$ for $i=1,\dots,m$.

Before we formulate our result on the indecomposability of the objects $C_n$, we recall the notion of a trivial extension. Let $R$ be a ring and $M$ an $(R^\op,R)$-bimodule. The {\em trivial extension} of $R$ by $M$ is the ring $R\ltimes M$ that has as underlying set the cartesian product $R\times M$, addition given componentwise, and multiplication defined by the formula
$$(r,m)(r',m'):=(rr',mr'+rm') \text{ for all }r,r'\in R\text{ and }m,m'\in M.$$
Equivalently, a trivial extension of $R$ by $M$ is characterised by a short exact sequence
$$0\to M\to E\to R\to 0$$
of $k$-modules such that $E\to R$ is an algebra homomorphism that splits and the kernel $M$ is an ideal in $E$ of square zero.

\begin{lem}\label{lem:local} A trivial extension of a local ring is again local.
\end{lem}
\begin{proof}
An element $(r,m)$ is a unit in $R\ltimes M$ if and only if $r$ is a unit in $R$. Hence the non-units in $R\ltimes M$ are the elements $\{(r,m)\mid r\in R\text{ non-unit},~m\in M\}$. They form an ideal in $R\ltimes M$ if and only if the non-units in $R$ form an ideal.
\end{proof}

\begin{prop}\label{prop:local} For all objects $C_n$, $n\in\mathbb N$ that have been constructed via the construction steps of the first and second kind as described before there is a $k$-algebra isomorphism $$\End_{\D(A)}(C_n)\cong\End_{\D(A)}(C_{n-1})\ltimes\Ho^{e_n^A-1}A.$$ In particular, $\End_{\D(A)}(C_n)$ is local, hence $C_n$ is indecomposable.
\end{prop}
\begin{proof}
We consider the triangle
$$\Sigma^{-e_n}A\overset\p\to C_{n-1}\to C_n\to\Sigma^{-e_n+1}A$$
and the corresponding long exact $\Hom$-sequence (instead of $\Hom_{\D(A)}(-,-)$ we write $(-,-)$ for short)
$$\cdots \to(\Sigma C_{n-1},C_n) \to (\Sigma^{-e_n+1}A,C_n) \to (C_n,C_n) \to (C_{n-1},C_n) \to (\Sigma^{-e_n}A,C_n) \to \cdots.$$
Since we are dealing just with homotopically projective objects, we can think of the morphism spaces to be in $\mathbf H(A)$ and do not have to consider fractions. First we notice that
\begin{align*}
(\Sigma C_{n-1},C_n) \to (\Sigma^{-e_n+1}A,C_n), \\
(C_{n-1},C_n) \to (\Sigma^{-e_n}A,C_n)
\end{align*}
are the zero maps. This can be seen as follows. Let $a$ be a morphism $\Sigma C_{n-1}\to C_n$. We want to show that the composition $a\circ(\Sigma\p)\colon\Sigma^{-e_n+1}A\overset{\Sigma\p}\to\Sigma C_{n-1}\overset a\to C_n$ is zero. We have that
$$(a\circ(\Sigma\p))(1_{\Sigma^{-e_n+1}A})=a(\zeta)\in C_n^{e_n-1}$$
and that $a(\zeta)=\zeta\cdot a(1_{\Sigma^{-e_{n-1}+2}A})$ with $\zeta\in A^{e_n-1},~a(1)\in C_n^{e_{n-1}-2}$.
Since $$C_n^{e_n-1}=((\Sigma^{-e_n+1}A)\oplus(\Sigma^{-e_{n-1}+1}A)\oplus(\Sigma^{-e_{n-2}+1}A))^{e_n-1}$$
and
$$C_n^{e_{n-1}-2}=((\Sigma^{-e_{n-2}+1}A)\oplus(\Sigma^{-e_{n-3}+1}A)\oplus(\Sigma^{-e_{n-4}+1}A))^{e_{n-1}-2},$$
we know that $a(\zeta)=\zeta\cdot a(1_{\Sigma^{-e_{n-1}+2}A})\in(\Sigma^{-e_{n-2}+1}A)^{e_n-1}=A^{d+e-2}$.
If $e\neq 2$, then $d+e-2>d$, hence $A^{d+e-2}=0$ and $a(\zeta)=0$.
If $e=2$, then $d\leq 2e=4$ and hence $d=4$. We claim that all elements in $A^{d+e-2}=A^4$ are boundaries in $C_n$. In particular, $a(\zeta)$ is a boundary, and hence the composition is clearly zero.
This is clear if $e_{n-1}^A=d$.
If $e_{n-1}^A=e$, then the Poincar\'e duality gives that multiplication with $\bar\zeta$ gives a non-zero map $\Ho^e A\to\Ho^4 A$. This map is then surjective because $\Ho^4A$ is one dimensional. Since the differential on $A^4$ is multiplication with $\zeta$ the assertion follows.

Now, let $a$ be a morphism $C_{n-1}\to C_n$. We want to show that the composition 
$$a\circ\p\colon\Sigma^{-e_n}A\overset\p\to C_{n-1}\overset a\to C_n$$ is zero. We have that
$$(a\circ\p)(1_{\Sigma^{-e_n}A})=a(\zeta)\in C_n^{e_n}$$
and that $a(\zeta)=\zeta\cdot a(1_{\Sigma^{-e_{n-1}+1}A})$ with $\zeta\in A^{e_n-1},~a(1)\in C_n^{e_{n-1}-1}$.
Since $$C_n^{e_n}=((\Sigma^{-e_{n+1}+1}A)\oplus(\Sigma^{-e_{n}+1}A)\oplus(\Sigma^{-e_{n-1}+1}A))^{e_n}$$
and
$$C_n^{e_{n-1}-1}=((\Sigma^{-e_{n-1}+1}A)\oplus(\Sigma^{-e_{n-2}+1}A)\oplus(\Sigma^{-e_{n-3}+1}A))^{e_{n-1}-1},$$
we know that $a(\zeta)=\zeta\cdot a(1_{\Sigma^{-e_{n-1}+1}A})\in(\Sigma^{-e_{n-1}+1}A)^{e_n}=A^{e_n^A}$. Since $a(1_{\Sigma^{-e_{n-1}+1}A})$ is just a scalar, $a(\zeta)$ is a boundary in $C_n$.


So far we have a short exact sequence
$$0 \to (\Sigma^{-e_n+1}A,C_n) \to (C_n,C_n) \to (C_{n-1},C_{n}) \to 0$$
We further claim that
\begin{enumerate}
\item $(C_{n-1},C_{n})\cong(C_{n-1},C_{n-1}),$
\item $(\Sigma^{-e_n+1}A,C_n)\cong(\Sigma^{-e_n+1}A,C_{n-1})\cong (\Sigma^{-e_n^A+1}A,A).$ 
\end{enumerate}
The isomorphism in (1) follows from the long exact $\Hom$-sequence
$$\xymatrix@=5mm{
\cdots\ar[r] & (C_{n-1},\Sigma^{-e_n}A)\ar[r]\ar@^{=}[d] & (C_{n-1},C_{n-1}) \ar[r] & (C_{n-1},C_n) \ar[r] & (C_{n-1},\Sigma^{-e_n+1}A) \ar[r]\ar@^{=}[d] & \cdots\\
& 0 & & & 0 &
}.$$
The first isomorphism in (2) is by Lemma~\ref{lem:Ho}(i), and the second one holds since $$C_{n-1}^{e_n-1}=((\Sigma^{-e_{n-1}+1}A)\oplus(\Sigma^{-e_{n-2}+1}A))^{e_n-1}=A^{e_n^A-1}\oplus A^{d+e-2}$$
and all elements of $A^{d+e-2}$ are boundaries as already shown and used above.
Hence we get a short exact sequence of $k$-modules
$$0 \to (\Sigma^{-e_n^A+1}A,A) \to (C_n,C_n) \to (C_{n-1},C_{n-1}) \to 0.$$
The morphism $(C_n,C_n) \to (C_{n-1},C_{n-1})$ is a $k$-algebra homomorphism that splits.
The multiplicativity can be seen from the following commutative diagram.
$$
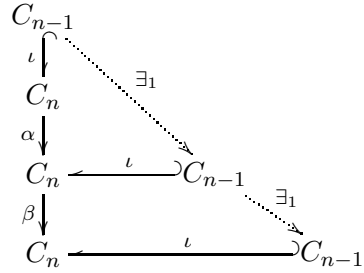
\xymatrix@=5mm{
C_{n-1} \ar@^{{(}->}_\iota[d]\ar@{.>}^{\exists_1}[rrdd]\\
C_n\ar[d]_\a \\
C_n \ar[d]_\b& & C_{n-1}\ar@{.>}^{\exists_1}[rd]\ar@_{{(}->}_\iota[ll]\\
C_n & & & C_{n-1}\ar@_{{(}->}_\iota[lll]\\
}.$$
The splitting is given by assigning to $a\colon C_{n-1}\to C_{n-1}$ the extended map that is determind by 
$$1_{\Sigma^{-e_n+1}A}\mapsto \a\cdot 1_{\Sigma^{-e_n+1}A}\in\Sigma^{-e_n+1}A\subseteq C_n^{e_n-1},$$
where $\a=a(1_A)\in k=A^0=C_{n-1}^0$. One checks that this gives a chain map by looking at the two cases $e_n^A=d$, and $e_n^A=e,~e_{n-1}^A=d$. The image of $(\Sigma^{-e_n^A+1}A,A) \to (C_n,C_n)$ is an ideal in $(C_n,C_n)$ of square zero, as one sees from the following commutative diagrams:
The first diagram shows that the square is zero.
$$
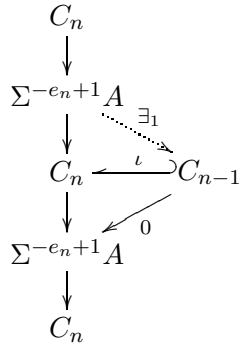
\xymatrix@=5mm{
C_n\ar[d]\\
\Sigma^{-e_n+1}A\ar[d]\ar@{.>}^{\exists_1}[rd]\\
C_n\ar[d] & C_{n-1}\ar@_{{(}->}_\iota[l]\ar^0[ld]\\
\Sigma^{-e_n+1}A\ar[d]\\
C_n \\
}$$
It is clear that it is a left ideal.
To show that it is a right ideal, we show that
the composition of a given morphism $C_n\to\Sigma^{-e_n+1}A\to C_n$ with some $a\colon C_n\to C_n$ from the right $C_n\overset{a}\to C_n\to\Sigma^{-e_n+1}A\to C_n$ is in the kernel of $(C_n,C_n)\to(C_{n-1},C_{n-1})$. This follows from the following diagram.
$$\xymatrix@=5mm{
C_{n-1}\ar@^{{(}->}[d]\ar@{.>}^{\exists_1}[rdd]\\
C_n\ar_a[d]\\
C_n\ar[d] & C_{n-1}\ar@_{{(}->}_\iota[l]\ar^0[ld]\\
\Sigma^{-e_n+1}A\ar[d]\\
C_n
}$$
Hence $(\Sigma^{-e_n^A+1}A,A)$ has a $(C_{n-1},C_{n-1})$-module structure. Therefore, $(C_n,C_n)$ is the trivial extension of $\Lambda:=(C_{n-1},C_{n-1})$ via the $(\Lambda^\op,\Lambda)$-bimodule $(\Sigma^{-e_n^A+1}A,A)$. Since a trivial extension of a local ring is again local (see Lemma~\ref{lem:local}) and we start with the local ring $(A,A)\cong k$, also $(C_n,C_n)$ is local and so $C_n$ indecomposable.
\end{proof}

\begin{rem}
We could alternatively use \cite[Lemma~6.5]{HKR} to get the indecomposability of the objects $C_n$. For applying this Lemma it is necessary to have $(C_{n-1},\Sigma^{-n+1}A)=0$. This has been used in our proof, too. But here in our more special situation, we do more and give a formula for the endomorphism ring of the objects $C_n\in\D(A)$.
\end{rem}

\subsection{The proof}
Now we are ready to finish the proof of Theorem~\ref{thm:noc} by giving objects that lie in different components:

ad (1): We denote by $C_\a$, $\a\in\{0,1\}^n$, $n\in\mathbb N_0$ the object that is constructed by the described $n$-step mapping cone construction with construction steps of the first and second kind encoded in the $n$-tuple $\a$. By the considerations above $\a$ should not contain two neighbour entries $1$. Denote by $\M_n\subseteq\{0,1\}^n$ the subset of all such $n$-tuples with no neighbour entries $1$. An example is $C_{(1,0,1)}$ that corresponds to some object $C_3$ that has been constructed using $e_1=e,~e_2=e+d-1,~e_3=2e+d-2$.  
The collection of objects can be visualised in the following way, where the arms going up mean a construction step of the first kind and the arms going down a construction step of the second kind.
$$\xymatrixrowsep{.02pc}\xymatrixcolsep{.4pc}
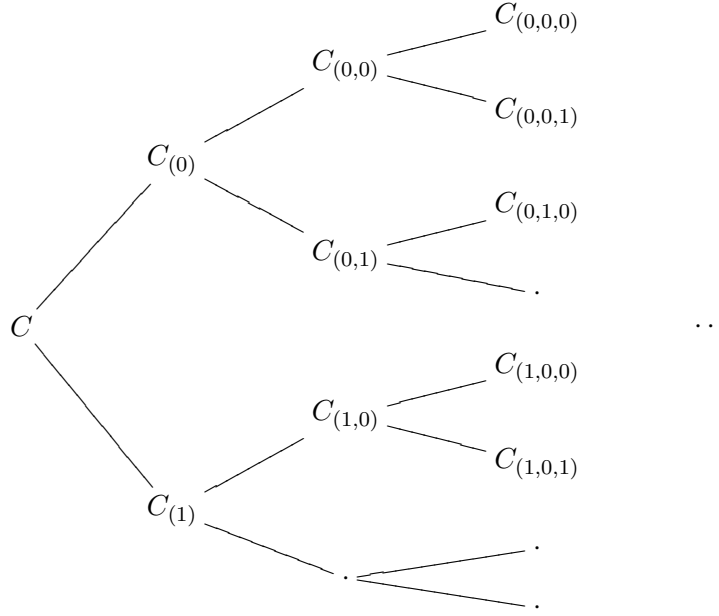
\xymatrix{ &&&&&&&&&& & & C_{(0,0,0)} \\
 &&&&&&& &  C_{(0,0)} \ar@{-}[rrrru]\ar@{-}[rrrrd] && \\
 &&&&&&&& & & && C_{(0,0,1)} \\
 &&&& C_{(0)} \ar@{-}[rrrruu]\ar@{-}[rrrrdd] & &&& \\
 &&&&&&&& &&& & C_{(0,1,0)}  \\
 &&&&&&&&  C_{(0,1)} \ar@{-}[rrrru] \ar@{-}[rrrrd] && \\
 &&&& &&&&&&&& \cdot && \\
 C \ar@{-}[rrrruuuu]\ar@{-}[rrrrdddd] & & & &&&&&&&&&&&&& \cdots\\
 &&&&&&&&&& & & C_{(1,0,0)} \\
 &&&&&&& &  C_{(1,0)} \ar@{-}[rrrru]\ar@{-}[rrrrd] && \\
 &&&&&&&& & & && C_{(1,0,1)} \\
 &&&& C_{(1)} \ar@{-}[rrrruu]\ar@{-}[rrrrdd] & &&& \\
 &&&&&&&& &&& & \cdot  \\
 &&&&&&&& \cdot \ar@{-}[rrrru]\ar@{-}[rrrrd] && \\
 &&&& &&&&&&&& \cdot \\
}$$
Given $n\in\mathbb N_0$, the $C_\a$, $\a\in\M_n$ are pairwise non isomorphic indecomposable and not shifts of each other as we see for example by looking at their cohomology (Lemma~\ref{lem:Ho}(ii)). They are compact and have the same value under the additive function $f$. Since objects in the same component with the same value under $f$ are just iterated $(d-1)$-shifts of each other the constructed modules have to lie in different Auslander-Reiten components. If $n$ increases, then also the number of elements of $\M_n$ increases. Hence, if $\dim\Ho^*A\geq 3$, then there are at least countably infinitely many components. 

ad (2): We have the additional assumption that $\dim \Ho^eA\geq 2$ for some $2\leq e\leq d-2$.
Let $\zeta_1,\zeta_2\in A^e$ such that $0\neq\bar\zeta_1,\bar\zeta_2\in \Ho^eA$ and $\bar\zeta_1,\bar\zeta_2$ linearly independent. Let $\la:=(\la_1,\la_2)\in k^2\setminus\{0\}$ and define $\p_\la\colon\Sigma^{-e}A\to A,~1\mapsto\la_1\zeta_1+\la_2\zeta_2$. Let $C=C_{(\la)}=C_1$ be by definition the mapping-cone of $\p_\la$.
$$\Sigma^{-e}A\overset{\p_\la}\to A\to C_{(\la)}\to\Sigma^{-e+1}A$$
As in (1) the $C_{(\la)},\la\in k^2\setminus\{0\}$ are compact, indecomposable, and have the same value under the additive function $f$. Let $\la'\in k^2\setminus\{0\}$ and $C':=C_{(\la')}$. We claim that $C\cong C'$ if and only if $\la=\mu\la'$ for some $\m\in k\setminus\{0\}$, i.e.\ $\overline\la=\overline{\la'}\in\mathbb P_1(k)$. If $\overline\la=\overline{\la'}\in\mathbb P_1(k)$, then an isomorphism between $C$ and $C'$ could easily been written down. For the case that $\overline\la\neq\overline{\la'}\in\mathbb P_1(k)$, we consider the long exact $\Hom$-sequence (instead of $\Hom_{\D(A)}(-,-)$ we write $(-,-)$ for short)
$$ \cdots \to (C,C') \to (A,C') \overset{\p^*_\la}\to (\Sigma^{-e}A,C') \to \cdots.$$
We first show that
$$\p^*_\la:=\Hom_{\D(A)}(\p_\la,C')\colon\Hom_{\D(A)}(A,C')\to\Hom_{\D(A)}(\Sigma^{-e}A,C')$$
is injective. Note that $\Hom_{\D(A)}(A,C')\cong\Ho^0C'$ and $\Hom_{\D(A)}(\Sigma^{-e}A,C')\cong\Ho^eC'$. Since $\Ho^0C'\cong k$, a morphism $a\colon A\to C'$ is given by a scalar which we also denote by $a$. Let $a$ be non-zero. If $ 0=\p_\la^*(a)=a\circ\p_\la$ in $\D(A)$, then $\Ho^*(a\circ\p_\la)=0$ and therefore $\Ho^*(a\circ\p_\la)(\bar 1)=\overline{a(\la_1\zeta_1+\la_2\zeta_2)}=0\in\Ho^eC'$, i.e.\ $a(\la_1\zeta_1+\la_2\zeta_2)\in\Im d^{e-1}_{C'}$. It holds that $d^{e-1}_{C'}=[\p_{\la'},d_A^{e-1}]\colon k\oplus A^{e-1}\to A^e$. So there exists $b\in k$ such that $a(\la_1\zeta_1+\la_2\zeta_2)-\p_{\la'}(b)\in\Im d_A^{e-1}$, i.e.\ $(a\la_1-b\la_1')\overline{\zeta_1}+(a\la_2-b\la_2')\overline{\zeta_2}=0\in\Ho^eA$. It follows $a\la_1-b\la_1'=0$ and $a\la_2-b\la_2'=0$. Hence $\la_1={\frac b a}\la_1'$ and $\la_2={\frac b a}\la_2'$, i.e.\ $\la={\frac b a}\la'$. So $(A,C') \to (\Sigma^{-e}A,C')$ is injective for $\overline\la\neq\overline{\la'}\in\mathbb P_1(k)$. Hence $(C,C') \to (A,C')$ is the zero map. But if there would be an isomorphism $C\overset\cong\to C'$, then the composition $A\to C\to C'$ is non-zero giving a contradiction. Since in addition the $C_{(\la)}$, $\la\in\mathbb P_1(k)$ are not shifts of each other, as in (1) they all have to lie in different components and we get a $1$-parameter family of components. The described construction corresponds to the construction step of the second kind. After a  construction step of the first kind we can continue with a step of the second kind as before and get a $2$-parameter family and so on.
\end{proof}

\subsection{An example}
We illustrate the mapping cone constructions in the proof of Theorem~\ref{thm:noc} by the following example.
\begin{exm}\label{exm:S2S2}
Let the characteristic of $k$ be zero. Since $S^e\times S^d$ is as the product of two formal spaces again a formal space we have that $C^*(S^e\times S^d;k)$ is weakly equivalent by a series of quasi-isomorphisms to its cohomology differential graded algebra $\Ho^*(S^e\times S^d;k)$. By the K\"unneth formula this is isomorphic to the tensor product of differential graded algebras
$$\Ho^*(S^e;k)\otimes_k\Ho^*(S^d;k)\cong k[x]/(x^2)\otimes_k k[y]/(y^2)\cong
\begin{cases}
k[x,y]/(x^2,y^2) & \text{if $e$ or $d$ even}\\
E(x,y) & \text{if $e$ and $d$ odd}
\end{cases}
,$$ where $\deg x=e,\deg y=d$, and $E(x,y)$ is the exterior algebra in the two generators $x$ and $y$. Note that in both cases the appearing differential graded algebra is a finite dimensional model for $C^*(S^d\times S^d;k)$ in the sense of Remark~\ref{rem:model} that we denote by $A$.
For $e\neq d$ the underlying complex of the differential graded algebra $A$ is
$$\xymatrix@=4mm{ A\colon & k \ar[r] & 0\ar[r] & \cdots\ar[r] & 0\ar[r] & k\ar[r] & 0\ar[r] & \cdots\ar[r] & 0\ar[r] & k\ar[r] & 0\ar[r] & \cdots\ar[r] & 0\ar[r] & k\\
& 1 &&&& x &&&& y &&&& xy
},$$
and for $e=d$
$$\xymatrix@=4mm{ A\colon & k \ar[r] & 0\ar[r] & \cdots\ar[r] & 0\ar[r] & k^2\ar[r] & 0\ar[r] & \cdots\ar[r] & 0\ar[r] & k\\
& 1 &&&& x,y &&&& xy
}$$
In the following we take $e=d=2$. The described mapping cone constructions of the objects $C_1=C_{(\la)}$ and $C_2=C_1(3d-1)$ in the proof of Theorem~\ref{thm:noc}(2) may be visualised as follows.
The morphism of differential graded $A$-modules
$$\xymatrix@=3mm{
\Sigma^{-d}A  \ar[dd]_{\p_\la}       &&&&&&&&&&& k \ar[rr]\ar[dd]_{[\la_1\la_2]} && 0\ar[rr]\ar[dd] && k^2 \ar[rr]\ar[dd]^{\la_2\brack\la_1} && 0\ar[rr] && k &&&&& \\
\\
A  &&&&&&&  k\ar[rr]  && 0\ar[rr] && k^2\ar[rr] && 0\ar[rr] && k &&&&&   \\
}$$
gives as mapping cone the differential graded $A$-module
$$\xymatrix@=3mm{
&&&&&&&&& k \ar[rr]\ar[rrdd]_{[\la_1\la_2]} && 0\ar[rr]\ar[rrdd] && k^2 \ar[rr]\ar[rrdd]^{\la_2\brack\la_1} && 0\ar[rr] && k &&&&&\\
C_1 \\
&&&&&&& k\ar[rr] && 0\ar[rr] && k^2\ar[rr] && 0\ar[rr] && k &&&&&
}$$
The morphism
$$\xymatrix@=3mm{
 \Sigma^{-3d}A  \ar[ddd]_\p       &&&&&&&&&&&&  k \ar[rr]\ar[dd]^\p && 0\ar[rr] && k^2\ar[rr] && 0\ar[rr] && k \\
 \\
 &&&& k \ar[rr]\ar[rrdd] && 0\ar[rr]\ar[rrdd] && k^2 \ar[rr]\ar[rrdd] && 0\ar[rr] && k \\
C_1 \\
&& k\ar[rr] && 0\ar[rr] && k^2\ar[rr] && 0\ar[rr] && k
}$$
gives as mapping cone
$$\xymatrix@=3mm{
 &&&&&&&&&&& k \ar[rr]\ar[rrdd]^\p && 0\ar[rr] && k^2\ar[rr] && 0\ar[rr] && k && \\
\\
C_2 &&&&& k \ar[rr]\ar[rrdd] && 0\ar[rr]\ar[rrdd] && k^2 \ar[rr]\ar[rrdd] && 0\ar[rr] && k \\
\\
&&& k\ar[rr] && 0\ar[rr] && k^2\ar[rr] && 0\ar[rr] && k
}$$

In Theorem~\ref{thm:noc}(2) we have proven that the $C_{(\la)},\la\in\mathbb P_1(k)$ give a 1-parameter family of components for the Auslander-Reiten quiver of $S^d\times S^d$. If one continues the construction on $C_2$  with a construction step as the first one, then we get a 2-parameter family and so on.
\end{exm}

Unfortunately, we were not able to distinguish the case were $\dim\Ho^nA\leq 1$ for all $n$ from the case were $\dim\Ho^eA\geq 2$ for some $e$. See Section~6 `Some open questions' for further discussion.

\newpage

\section{Another interpretation of the number of components}

We give an interpretation of the number of components of the Auslander-Reiten quiver of a triangulated Krull-Remak-Schmidt category having Auslander-Reiten triangles in terms of equivalence classes of simple objects in the abelianisation of the given triangulated category. This is based on the functorial approach to Auslander-Reiten triangles given in Auslander's Philadelphia notes \cite{A2} (see also Appendix~B). 

Let $\T$ be a triangulated Krull-Remak-Schmidt category having Auslander-Reiten triangles, e.g.\ $\T=\D^c(A)$ for $A$ a simply connected Gorenstein differential graded algebra of finite type. Further let $\widehat\T:=\fp(\T^\op,\Ab)$ be the functor category of finitely presented additive functors $\T^\op\to\Ab$, where a functor $F$ is {\em finitely presented} (or {\em coherent}) if there exists an exact sequence $(-,Y)\to(-,X)\to F\to 0$. We define for simple objects $S,T\in\widehat\T$ the relation
$$S\sim T \iff \Ext^1_{\widehat\T}(S,T)\neq 0\vee\Ext^1_{\widehat\T}(T,S)\neq 0$$
and take the induced equivalence relation.
\sloppy
\begin{prop}\label{prop:func}
Let $\T$ be a triangulated Krull-Remak-Schmidt category having Auslander-Reiten triangles. Then there is a natural one-to-one correspondence between the Auslander-Reiten components of $\T$ and the equivalence classes of simple objects in $\widehat\T$.
\end{prop}
\fussy
The proposition is a consequence of the fact that there is a one-to-one correspondence between the isomorphism classes of indecomposable objects in $\T$ and the isomorphism classes of simple objects in $\widehat\T$ given by the assignment $$X\mapsto S_X:=\Hom_\T(-,X)/{\rad\Hom_\T(-,X)}$$ (see Proposition~\ref{prop:simplefunct}) and the following lemma.

\begin{lem}
Let $X,X'$ be non-isomorphic indecomposable objects in $\T$. Then it holds $\Ext^1_{\widehat\T}(S_X,S_{X'})\neq 0$ if and only if there exists an irreducible map $X'\to X$.
\end{lem}
\begin{proof}
Let $\tau X\overset f\to Y\overset g\to X\to\Sigma\tau X$ be an Auslander-Reiten triangle. Then 
$$(-,\tau X)\overset{f_*}\to (-,Y)\overset{g_*}\to (-,X)\overset\pi\to S_X\to 0$$
is a minimal projective presentation of $S_X$. Denote the complex
$$(-,\tau X)\overset{f_*}\to (-,Y)\overset{g_*}\to (-,X)$$ by $P_*$ and apply $(-,S_{X'})$ to it. Then $\Ext^1_{\widehat\T}(S_X,S_{X'})=\Ho^1(P_*,S_{X'})$ and $\Ext^1_{\widehat\T}(S_X,S_{X'})\neq 0$ means that there is a non-zero natural transformation $\phi\colon(-,Y)\to S_{X'}$ such that the composition $\phi\circ f_*$ is zero and $\phi$ does not factor through $g_*$.

Let $\phi\colon (-,Y)\to S_{X'}$ represent a non-zero element in $\Ext^1_{\widehat\T}(S_X,S_{X'})$. Since $(-,Y)$ and $(-,X')$ are projective, and $\pi'\colon(-,X')\to S_{X'}$ and $\phi\colon(-,Y)\to S_{X'}$ are epimorphisms there exist factorisations $(-,\a)\colon(-,Y)\to(-,X')$ and $(-,\b)\colon(-,X')\to(-,Y)$.
$$\xymatrix{
& (-,Y) \ar^\phi[d] \ar@{-->}_{(-,\a)}[dl] & (-,X') \ar^{\pi'}[ld] \ar@{-->}_{(-,\b)}[l]\\
(-,X') \ar^{\pi '}[r] & S_{X'}
  }$$
Since $\pi'\colon(-,X')\to S_{X'}$ is a projective cover it is right minimal hence the composition $(-,\a)\circ(-,\b)$ is an isomorphism. Therefore $(-,\b)$,  and hence $\b$ are split monomorphisms. So $X'$ is a direct summand of $Y$ and there exists an irreducible map $X'\to X$.

If on the other hand there is an irreducible map $a\colon X'\to X$, then there exists a split epimorphism $p\colon Y\to X'$. We claim that the composition $\phi:=\pi'\circ p_*$ gives a non-zero element in $\Ext^1_{\widehat\T}(S_X,S_{X'})$. Since the composition  $p\circ f\colon\t X\to X'$ is also an irreducible map, there exists a split monomorphism $j\colon\t X\to Y'$ such that $p\circ f=g'\circ j$. Hence we have the following commutative diagram
$$\xymatrix{
 & (-,\tau X) \ar^{f_*}[r]\ar^{j_*}[d] & (-,Y) \ar^{g_*}[r]\ar^{p_*}[d] & (-,X) \ar^\pi[r] & S_X \ar[r] & 0 \\
(-,\tau X') \ar^{f'_*}[r] & (-,Y') \ar^{g'_*}[r] & (-,X') \ar^{\pi '}[r] & S_{X'} \ar[r] & 0
  }$$
and we get
$$\phi\circ f_*=\pi '\circ p_*\circ f_*=\pi '\circ g'_*\circ j_*=0.$$
If there would be a factorisation
$$\xymatrix{
(-,Y) \ar^{g_*}[r]\ar_\phi[d] & (-,X) \ar^\psi@{-->}[ld] \\
S_{X'}
  }$$
then with the same arguments as used before we get factorisations
$$\xymatrix{
& (-,X) \ar^\phi[d] \ar@{-->}_{(-,\a)}[dl] & (-,X') \ar^{\pi'}[ld] \ar@{-->}_{(-,\b)}[l]\\
(-,X') \ar^{\pi '}[r] & S_{X'}
  }$$
and conclude that $X'$ is a direct summand of $X$. Since $X$ is indecomposable, we have $X'\cong X$ in contradiction to the existence of an irreducible map $X'\to X$.

\end{proof}

\newpage

\section{Some open questions}

In this section we discuss some questions related to the material in the previous sections that might be interesting to find answers for. 


%

\subsection{Discrete/wild dichotomy}

The following schema and the comments on it below give an overview on the connections between the Auslander-Reiten quiver and cohomology dimensions of a simply connected Gorenstein differential graded algebra $A$ of finite type that we would expect.

\begin{enumerate}
\item[(i)] $\dim\Ho^eA\leq 1\text{ for all } e\in\mathbb Z$
	\begin{enumerate}
	\item $\dim\Ho^*A=1:\quad\quad\text{discrete family of trivial components}$
	\item $\dim\Ho^*A=2:\quad\quad\text{finite family of }\mathbb ZA_\infty\text{ components}$
	\item $\dim\Ho^*A=3:\quad\quad\text{discrete family of }\mathbb ZA_\infty\text{ components }$
	\item $\dim\Ho^*A>3:\quad\quad \text{$n$-parameter families of }\mathbb ZA_\infty\text{ components for all }n$
	\end{enumerate}
\item[(ii)] $\dim\Ho^eA>1\text{ for some } e\in\mathbb Z\quad\text{$n$-parameter families of }\mathbb ZA_\infty\text{ components for all }n$
\end{enumerate}

The case (i.a) is the trivial simply connected case where the Auslander-Reiten quiver has countably many components each of them consisting of just a single vertex.

The case (i.b) is the case of the spheres that has been calculated by J\o rgensen in \cite{J1}.

The case (i.c): We have constructed a discrete family of $\mathbb ZA_\infty$-components in Theorem~\ref{thm:noc}. It is not clear to us whether this is really a discrete phenomenon or whether parameter families do exist? That it is a discrete phenomenon would be true if the following holds.
\begin{conj}
If $\dim\Ho^eA\leq 1$ for all $e\in\mathbb Z$ and $\dim\Ho^*A=3$, then for each $n\in\mathbb N$ there are only a finite number of non-isomorphic indecomposable objects in $\D^c(A)$ that have the value $n$ under the additive function $f$. In particular, the Auslander-Reiten quiver of $\D^c(A)$ consists only of countably many components.
\end{conj}

But even for the simplest differential graded algebra of this class $A=k[x]/(x^3)$ with $\deg x=2$ this conjecture is not clear to us. One somehow would have to classify the indecomposable compact objects over this differential graded algebra.

The case (i.d): We give some examples where we have constructed $n$-parameter families for all $n$. But we do not have a uniform way to construct such parameter families.

\begin{exm}
Consider the differential graded algebra $A=k[x,y]/(x^2,y^2)$ with $\deg x=2$ and $\deg y=4$. The underlying complex of $A$ is
$$\xymatrix@=4mm{
A\colon & k \ar[r] & 0 \ar[r] & k \ar[r] & 0 \ar[r] & k \ar[r] & 0 \ar[r] & k\\
& 1 && x && y && xy
}.$$
In characteristic zero this is a finite dimensional model for the singular cochain differential graded algebra of the product space of spheres $S^2\times S^4$ (see Example~\ref{exm:sphere}).
Now consider the following differential graded $A$-module.
$$\xymatrix@=4mm{
&&& k \ar[r]\ar[rd]^>>{\la_1}\ar[rddd]^{\la_2} & 0 \ar[r] & k \ar[r]\ar[rddd]^{\la_2} & 0 \ar[r] & k \ar[r]\ar[rd] & 0 \ar[r] & k\\
&& k \ar[r]\ar[rd] & 0 \ar[r] & k \ar[r] & 0 \ar[r] & k \ar[r]\ar[rd] & 0 \ar[r] & k\\
& k \ar[r]\ar[rd] & 0 \ar[r] & k \ar[r] & 0 \ar[r] & k \ar[r]\ar[rd] & 0 \ar[r] & k\\
k \ar[r] & 0 \ar[r] & k \ar[r] & 0 \ar[r] & k \ar[r] & 0 \ar[r] & k
}$$
The last mapping cone construction `on top' gives the possibility to get a $\mathbb P_1(k)$-parameter family of Auslander-Reiten components. An iteration of this construction gives a $2$-parameter family and so on. Hence for the Auslander-Reiten quiver of this algebra there exist $n$-parameter families of components.
\end{exm}

\begin{exm}
Consider the differential graded algebra $A=k[x]/(x^4)$ with $\deg x=2$.  The underlying complex of $A$ is
$$\xymatrix@=4mm{
A\colon & k \ar[r] & 0 \ar[r] & k \ar[r] & 0 \ar[r] & k \ar[r] & 0 \ar[r] & k\\
& 1 && x && x^2 && x^3
}.$$
In characteristic zero this is a finite dimensional model for the singular cochain differential graded algebra of the complex projective space $\mathbb CP^3$. This is true since $\mathbb CP^3$ is a formal space and hence $C^*(\mathbb CP^3;k)$ is weakly equivalent to its cohomology algebra $\Ho^*(\mathbb CP^3;k)$. 
Now consider the following differential graded $A$-module.
$$\xymatrix@=4mm{
&&&&&&&&&&& k \ar[r]\ar[rd]^>>{\la_1}\ar[rdd]^>>{\la_2} & 0 \ar[r] & k \ar[r]\ar[rd]^>>{\la_1} & 0 \ar[r] & k \ar[r] & 0 \ar[r] & k\\
&&&&&&&& k \ar[r]\ar[rddd] & 0 \ar[r] & k \ar[r] & 0 \ar[r] & k \ar[r] & 0 \ar[r] & k\\
&&&&&& k \ar[r]\ar[rd]\ar[rdd] & 0 \ar[r] & k \ar[r]\ar[rd]\ar[rdd] & 0 \ar[r] & k \ar[r]\ar[rd] & 0 \ar[r] & k\\
&&&&& k \ar[r]\ar[rdd] & 0 \ar[r] & k \ar[r] & 0 \ar[r] & k \ar[r] & 0 \ar[r] & k\\
&&& k \ar[r]\ar[rd] & 0 \ar[r] & k \ar[r]\ar[rd] & 0 \ar[r] & k \ar[r] & 0 \ar[r] & k\\
k \ar[r] & 0 \ar[r] & k \ar[r] & 0 \ar[r] & k \ar[r] & 0 \ar[r] & k
}$$
or the following minimal semi-free resolution of it.
$$\xymatrix@=4mm{
&&&&&&&&&&& k \ar[r]\ar[rd]^>>{\la_1}\ar[rdd]^>>{\la_2} & 0 \ar[r] & k \ar[r]\ar[rd]^>>{\la_1} & 0 \ar[r] & k \ar[r] & 0 \ar[r] & k\\
&&&&&&&& k \ar[r]\ar[rddd] & 0 \ar[r] & k \ar[r] & 0 \ar[r] & k \ar[r] & 0 \ar[r] & k\\
&&&&&& k \ar[r]\ar[rd]^>>1\ar[rdd]^>>2 & 0 \ar[r] & k \ar[r]\ar[rd]\ar[rdd] & 0 \ar[r] & k \ar[r]\ar[rd] & 0 \ar[r] & k\\
&&&&& k \ar[r] & 0 \ar[r] & k \ar[r] & 0 \ar[r] & k \ar[r] & 0 \ar[r] & k\\
&&& k \ar[r]\ar[rd] & 0 \ar[r] & k \ar[r]\ar[rd] & 0 \ar[r] & k \ar[r] & 0 \ar[r] & k\\
k \ar[r] & 0 \ar[r] & k \ar[r] & 0 \ar[r] & k \ar[r] & 0 \ar[r] & k
}$$
The last mapping cone construction `on top' gives the possibility to get a $\mathbb P_1(k)$-parameter family of Auslander-Reiten components. An iteration of this construction gives a $2$-parameter family and so on. Hence for the Auslander-Reiten quiver of this algebra there exist $n$-parameter families of components for all $n$.
\end{exm}

The case (ii): We have constructed $n$-parameter families of $\mathbf ZA_\infty$-components for all $n$ in Theorem~\ref{thm:noc}.

\subsection{Topological interpretation}
The connections of the structure of the Auslander-Reiten quiver to the cohomology dimensions that we expect have been discussed before. But we also ask if there are any further topological interpretations for the results on the structure of the Auslander-Reiten quiver? Does it makes sense to talk about `finite', `discrete'/`tame', `wild' spaces?

\subsection{More general differential graded algebras}
Let $A$ be a differential graded algebra that is not necessarily simply connected. What can one say about the existence of Auslander-Reiten triangles and the structure of the Auslander-Reiten quiver in that case?

\subsubsection{The Gorenstein property}
As already mentioned in Remark~\ref{rem:gor}(1)  for an ordinary $k$-algebra $A$ seen as differential graded algebra concentrated in degree zero the third condition in Proposition~\ref{prop:gor} says that $A$ and $A^\op$ have finite injective dimension. This is the definition of a Gorenstein algebra given by Auslander in \cite{A}. Furthermore, for an augmented simply connected differential graded algebra $A$ of finite type one equivalent condition in Theorem~\ref{thm:poincare} for the existence of Auslander-Reiten triangles in $\D^c(A)$ was the Gorenstein condition given by Avramov-Foxby in \cite{AF}. In Remark~\ref{rem:gorfj} we mentioned another Gorenstein definition given by Frankild-J\o rgensen in \cite{FJ}, namely that the functor $$\RHom_A(-,A)\colon\D^{f,b}(A)\to\D^{f,b}(A^\op)$$ is a duality. The question arises if some Gorenstein definition is related to the existence of Auslander-Reiten triangles also in a more general setup. Looking at the definition of Frankild-J\o rgensen we only see that the equivalent conditions in Proposition~\ref{prop:gor} imply that the functor $\RHom(-,A)\colon\D(A)\to\D(A^\op)$ restricts to $\D^f(A)$ and hence also to $\D^{f,b}(A)$:

\begin{lem}\label{lem:^f}
Let $A$ be a differential graded $k$-algebra and $DA\in\D^c(A^\op)$, then the functor $\RHom_A(-,A)\colon\D(A)\to\D(A^\op)$ restricts to a functor $\D^f(A)\to\D^f(A^\op)$.
\end{lem}
\begin{proof}
Let $M\in\D^f(A)$. From $\RHom_{A^\op}(A,DM)\cong DM\in\D^f(A^\op)$ it follows that the functor $\RHom_{A^\op}(-,DM)$ sends compact objects to objects in $\D^f(A^\op)$. Since $$\RHom_A(M,A)\cong\RHom_A(DDM,DDA)\cong\RHom_{A^\op}(DA,DM)$$
we get the assertion.
\end{proof}


\subsubsection{The structure of the Auslander-Reiten quiver}
What can we say about the structure of the Auslander-Reiten quiver of differential graded algebra $A$ such that $\D^c(A)$ has Auslander-Reiten triangles but $A$ is not necessary simply connected? For example let $A$ be a differential graded algebra as in Remark~\ref{rem:CY}, i.e.\ $\Ho^nA$ is finite dimensional over $k$ for all $n\in\mathbb N$ and $\D^c(A)$ is a Calabi-Yau category. In this situation one does not have the additive function $f=\dim\Ho^*\RHom_A(-,k)$ (see Section~3) that was used by J\o rgensen to analyse the structure of the Auslander-Reiten components in the simply connected case. One possible approach to replace it might be to consider the assignment $M\mapsto\lev(M)=\lev_{\D(A)}^A(M)$ taken from \cite{ABYM}. Here $\lev(M)$ counts the number of steps that are necessary to build $M$ from $A$ via triangles. More precisely, denote by $\thick^0_{\D(A)}(A)=\{0\}$ and by $\thick_{\D(A)}^1(A)$ the smallest strict full subcategory of $\D(A)$ which contains $A$ and is closed under taking finite coproducts, retracts, and all shifts. Inductively let $\thick_{\D(A)}^n(A)$ be the smallest strict full subcategory of $\D(A)$ which is closed under retracts and contains objects $X$ which admit an exact triangle
$$X_1\to X\to X_2\to \Sigma X_1$$
in $\D(A)$ with $X_1\in \thick_{\D(A)}^{n-1}(A)$ and $X_2\in \thick_{\D(A)}^1(A)$. Note that $\thick_{\D(A)}^n(A)$ is also closed under shifts and finite coproducts. Define
$$\lev\colon\D(A)\to\mathbb N_0\cup\{\infty\},\quad M\mapsto\lev(M):=\inf\{n\in\mathbb N_0\mid M\in \thick_{\D(A)}^n(A)\}.$$
This function has the following properties.
\begin{enumerate}
\item An object $M\in\D(A)$ is compact if and only if $\lev(M)<\infty$. This suggests to call an object $M\in\D^c(A)$ {\em finitely built} from $A$.
\item The function $\lev$ is subadditive on triangles. It is not clear to us whether it is additive on Auslander-Reiten triangles or not.
\item The function $\lev$ is bounded on the class of all (indecomposable) compact objects if and only if the dimension $\dimt\D^c(A)$ of the triangulated category $\D^c(A)$ is finite. Here the dimension of a triangulated category $\T$ as defined by Rouquier in \cite{Rq} is $$\dimt\T:=\inf\{n\mid\exists X\in\T \text{ s.t. }\lev_{\T}^X(M)\leq n \text{ for all } M\in\T\}.$$
\end{enumerate}
\begin{proof}
(1) $\D^c(A)=\langle A\rangle_{\rm thick}=\bigcup_{n\in\mathbb N_0}\thick_{\D(A)}^n(A)$\\
(2) One uses the octahedral axiom.
\\
(3) If $\dimt\D^c(A)=n<\infty$, then there exists some $X\in\D^c(A)$ such that
 all $M\in\D^c(A)$ can be built from $X$ in less or equal $n$ steps. Since we have $\lev_{\D(A)}^A(X)<\infty$ and 
$$\lev_{\D(A)}^A(M)\leq\lev_{\D(A)}^X(M)\lev_{\D(A)}^A(X)\leq n\lev_{\D(A)}^A(X)$$
for all $M\in\D^c(A)$ we have $n\cdot\lev_{\D(A)}^A(X)$ as a bound for $\lev_{\D(A)}^A$ restricted to $\D^c(A)$. If on the other hand $\lev_{\D(A)}^A$ is bounded on the set of all indecomposable compact objects by some positive integer $n$, then by definition it is also bounded by $n$ on the class of all compact objects. Then $\dimt\D^c(A)\leq n$ is finite.
\end{proof}

In the simply connected situation where we have the function $f$ it holds

\begin{lem}
Let $A$ be an augmented simply connected differential graded algebra of finite type, and $M\in\D^f(A)$ with $\inf\Ho^*M>-\infty$, then $f(M)\geq\lev_{\D(A)}^A(M)$.
\end{lem}
\begin{proof}
If $f(M)$ is finite, then there exists a minimal semi-free resolution $F\to M$ as in Lemma~\ref{lem:sfr} with semi-free filtration of $F$ that terminates after finitely many steps. Then $M$ can be built from $A$ in less or equal than $f(M)$ steps.
\end{proof}

Even in the simply connected situation we do not know whether the function $\lev$ could equally be considered instead of $f$ or not. We would like to know whether $\lev$ is additive on Auslander-Reiten triangles or unbounded on Auslander-Reiten components. For the singular cochain differential graded algebra of a sphere we are able to show this: Recall the notations $A=k[x]/(x^2)$ with $\deg x=d$, $B=k[y]$ with $\deg y=-d+1$, and the equivalence $\D^c(A)\simeq\langle k\rangle_{\thick}\subseteq\D(B)$ from Example~\ref{exm:sphere}. Further recall from \cite{J1} that $\Sigma^jB/(y^m)$, $j\in\mathbb Z,~m\in\mathbb N_0$ are representatives of the isomorphism classes of indecomposable objects in $\langle k\rangle_{\thick}\subseteq\D(B)$. Note that under the equivalence these objects are mapped to the objects $C_n$ in $\D^c(A)$ that are built via the $n$-times iterated mapping cone constructions of the first kind as described in Section~4. For example $C_2$ is the object that may be visualised as follows.
$$\xymatrix@=4mm{
&&&&&& k \ar[r]\ar[rd] & 0 \ar[r] & \cdots \ar[r] & 0 \ar[r] & k \\
&&& k \ar[r]\ar[rd] & 0 \ar[r] & \cdots \ar[r] & 0 \ar[r] & k \\
k \ar[r] & 0 \ar[r] & \cdots \ar[r] & 0 \ar[r] & k
}$$


\begin{prop}\label{prop:lev}
Let $A=k[x]/(x^2)$ with $\deg x=d$ and the $C_n$ be the representatives of the isomorphism classes of indecomposable objects in $\D^c(A)$ described above. Then
$$\lev(C_n)=n+1\text{ for all } n\in\mathbb N_0.$$
In particular, the functions $\lev$ and $f$ coincide on indecomposable compact objects in $\D(k[x]/(x^2))$, and $\lev$ is an additive unbounded function on each Auslander-Reiten component of $\D^c(k[x]/(x^2))$.
\end{prop}
Before we give the proof we first have a look at some results from \cite{ABYM}. I thank Srikanth Iyengar for advising me some of the methods in this article. Let $M$ and $C$ be differential graded $A$-modules. A {\em $C$-filtration} of $M$ is a filtration of submodules 
$$0=M^0\subset M^1\subset\cdots\subset M^l=M$$
with consecutive quotients isomorphic to a direct sum of shifts of copies of $C$. The $C$-filtration is called {\em finite} if in the quotients just finite direct sums are involved. If the module $C$ is homotopically projective, then with the same proof as \cite[Theorem~4.2]{ABYM} we have that $\lev_{\D(A)}^C(M)\leq l$ if and only if $M$ is retract of some differential graded $A$-module with finite $C$-filtration of length less or equal than $l$. 

\begin{proof}[proof of Proposition~\ref{prop:lev}]
Using the equivalence and remarks above we calculate
$$\lev_{\D(A)}^A(C_n)=\lev_{\D(B)}^k(C_n\otimes^L_Ak)=\lev_{\D(B)}^k(B/(y^{n+1}))=n+1.$$
Here the last equation holds since we know how the submodules of the uniserial differential graded $B$-module $B/(y^n)$ look like.
\end{proof}

As I learned from Steffen Oppermann, another possibility for proving the previous lemma is to use the following version of \cite[Lemma~4.11]{Rq} (see also \cite[Lemma~3.1]{Op}).
\begin{lem}\label{lem:Op}
Let $A$ be a differential graded algebra and
$$M\to N_0\overset{f_1}\to N_1\overset{f_2}\to\cdots\overset{f_d}\to N_d$$
be a sequence of morphisms in $\D(A)$ such that the composition $f_d\circ f_{d-a}\circ\cdots\circ f_1$ is non-zero, and $\Ho^*f_i=0$ for all $i=1,\cdots,d$. Then $\lev(M)>d$.
\end{lem}

Given the object $C_n\in\D^c(k[x]/(x^2))$ from above we consider the canonical morphism
$$\psi\colon A\oplus\Sigma^{-\sup C_n}A\to C_n,$$
and take the mapping cone $C(\psi)$.
$$A\oplus\Sigma^{-\sup C_n}A\overset\psi\to C_n\overset\pi\to C(\psi)\to\Sigma(A\oplus\Sigma^{-\sup C_n}A)$$
Note that for all $i\in\mathbb Z$ any non-zero morphism $\Sigma^iA\to C_n$ factors through $\psi$, hence $\Ho^*\pi=0$. The object $C(\psi)$ is quasi-isomorphic to $\Sigma^{-d+1}C_n$. We take $f_1$ to be the composition
$$f_1\colon C_n\overset\pi\to C(\psi)\overset\cong\to\Sigma^{-d+1}C_n.$$
Then we continue in the same way with the shifted object $\Sigma^{-d+1}C_n$ and construct $$f_2\colon\Sigma^{-d+1}C_n\to\Sigma^{-2d+2}C_n$$
and so on. One calculates $f_n\circ f_{d-1}\circ\cdots\circ f_1\neq 0$ and gets $\lev(C_n)>n$ from Lemma~\ref{lem:Op}.



\newpage
\appendix

\section{The derived category of a differential graded algebra}

In this appendix we give some background material about differential graded algebras and their derived categories. For the convenience of the reader a few proofs are included. For further details we refer to \cite{Ke}, \cite{AFH} and \cite{K2}.

\subsection{Definition of a differential graded algebra and its derived category}

Let $k$ be a field. 
Let $\GAlg$ be the category of $\mathbb Z$-graded algebras over $k$. Let $A\in\GAlg$. For a {\em homogeneous} element $a\in A^n$ we denote $|a|$ for its {\em degree} $n$.

A {\em differential graded algebra} over $k$ is a $\mathbb Z$-graded algebra $A=\coprod_{i\in\mathbb Z}A^i$ together with a homogeneous $k$-linear differential $d\colon A\to A$ (i.e.\ $d$ is of degree $+1$ and $d^2=0$) satisfying the {\em Leibniz rule}
$$d(ab)=d(a)b+(-1)^nad(b), \text{ for all elements } a\in A^n,b\in A.$$

\begin{exm}\label{exm:endom}
(1) Let $R$ be a ring. The endomorphism differential graded algebra $\E nd_R(X)$ of a complex of $R$-modules $X$ is the complex that consists in degree $i$ of the set of $R$-linear homomorphisms of graded modules $\Sigma^{-i}X\to X$, and differential given by $df:=d\circ f-(-1)^if\circ d$. Together with multiplication by composition this gives a differential graded algebra.\\
(2) Let $V$ be a graded $k$-vectorspace. The (graded) tensor algebra $TV$ of $V$ is defined as
$$TV=\coprod_{n\in\mathbb N_0}V^{\otimes n}$$
where $V^{\otimes n}$ is the $n$-times tensor product of $V$ over $k$. The degree of some element $x_1\otimes x_2\otimes\cdots\otimes x_n\in V^{\otimes n}$ where the $x_i$ are homogeneous elements in $V$ is the sum of the degrees of the elements $x_i\in V$, and the multiplication in $TV$ is given by $x\cdot y:=x\otimes y$ for $x,y\in TV$. Any differential $d\colon V\to TV$ extends uniquely to a differential on $TV$ by defining on homogeneous elements
$$d(x_1\otimes\cdots\otimes x_n):=\sum_{i=1}^n (-1)^{|x_1\otimes\cdots\otimes x_{i-1}|} x_1\otimes\cdots\otimes d(x_i)\otimes\cdots\otimes x_n.$$
Hence any complex $V$ of $k$-vectorspaces gives rise to a differential graded tensor algebra $TV$.\\
(3) Consider the differential graded tensor algebra $TV$ from (2). The elements
$$x\otimes y-(-1)^{|x||y|}y\otimes x,\text{ with }x,y\in V\text{ homogeneous}$$
generate a {\em differential graded ideal} $I$ of $TV$, i.e.\ a subspace that is closed under multiplication with elements from $TV$ from the left and the right, and closed under applying the differential. The quotient differential graded algebra $FV:=TV/I$ is the {\em free commutative graded algebra}.\\
(4) 
Let $V$ be a $k$-vectorspace viewed as graded vectorspace concentrated in degree $-1$ and $d\colon V\to k$ a $k$-linear map. Consider the differential graded tensor algebra $TV$ from (2). The elements $x\otimes x,~x\in V$ generate a differential graded ideal $J$ of $TV$ and the quotient differential graded chain algebra $EV=TV/{J}$ is the {\em exterior algebra}. If $k$ is more generally a commutative ring, then for example the Koszul complex of some given elements in $k$ can be interpreted as exterior algebra. Alternatively one could take a graded vectorspace $V$ without differential and gets the exterior algebra $EV$ with trivial differential.\\
(5) The singular cochain differential graded algebra of a topological space (see Appendix~C).
\end{exm}

The {\em opposite} differential graded algebra $A^\op$ of a differential graded algebra $A$ is the differential graded algebra with the same structure as $A$ but the multiplication is given by
$$a\cdot_\op a'=(-1)^{nm}a'a,\text{ for all elements } a\in A^n,a'\in A^m.$$
We call a differential graded algebra $A$ {\em (graded) commutative} if $A=A^\op$, i.e.\ 
$$aa'=(-1)^{nm}a'a,\text{ for all elements } a\in A^n,a'\in A^m.$$

A {\em morphism} of differential graded algebras is a morphism of the underlying graded algebras of degree zero that commutes with the differential. We denote the category of differential graded algebras by $\DGAlg$ and the category of commutative differential graded algebras by $\CDGAlg$.  A morphism of differential graded algebras is a {\em quasi-isomorphism} if it is an isomorphism in cohomology. We say that two differential graded algebras $A,A'$ are {\em weakly equivalent} in $\DGAlg$ (respectively $\CDGAlg$) if there is a zig-zag of quasi-isomorphisms in $\DGAlg$ (respectively $\CDGAlg$) between them.
$$
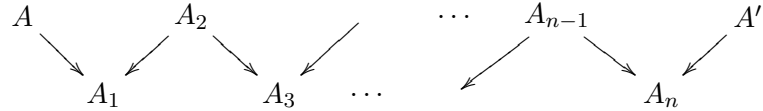
\xymatrix@=5mm{
A \ar[rd] && A_2 \ar[rd]\ar[ld] && \ar[ld] & \dots & A_{n-1} \ar[rd]\ar[ld]  && A' \ar[ld] \\
& A_1 && A_3 & \dots &&& A_{n}
}$$

A {\em differential graded (right) module} over a differential graded algebra $A$ is a graded right module $M=\coprod_{i\in\mathbb Z}M^i$ over the underlying graded algebra of $A$ together with a homogeneous $k$-linear differential $d\colon M\to M$ satisfying the {\em Leibniz rule}
$$d(xa)=d(x)a+(-1)^nxd(a), \text{ for all elements } x\in M^n,a\in A.$$

A {\em differential graded left module} over a differential graded algebra $A$ is a graded left module $M=\coprod_{i\in\mathbb Z}M^i$ over the underlying graded algebra of $A$ together with a homogeneous $k$-linear differential $d\colon M\to M$ satisfying the {\em Leibniz rule}
$$d(ax)=d(a)x+(-1)^nad(x), \text{ for all elements } x\in M,a\in A^n.$$
Equivalently, a differential graded left $A$-module $M$ corresponds to a morphism of differential graded algebras $A\to\E nd_k(M)$, where $M$ is a complex of $k$-vectorspaces.
Differential graded left $A$-modules can be identified with differential graded $A^\op$-modules via the sign convention $$ax=(-1)^{mn}xa, \text{ for all elements } x\in M^m,a\in A^n.$$

A morphism $M\to N$ of differential graded $A$-modules is a morphism of degree $0$ of the underlying graded modules over the underlying graded algebra of $A$ that commutes with the differential.
We denote the category of all differential graded $A$-modules by $\mathbf C(A)$. 

Let
$$(-)^\natural\colon \DGAlg\to\GAlg$$
be the forgetful functor that assigns to a differential graded algebra its underlying graded algebra and denote in the same way the forgetful functor that assigns to a differential graded module over a differential graded algebra its underlying graded module over the underlying graded algebra
$$(-)^\natural\colon \mathbf C(A)\to\mathbf C_{gr}(A^\natural).$$
By a morphism $f\colon M\to N$ of degree $m$ of graded modules $M,N$ over a graded algebra $A$ we mean a $k$-linear homogeneous map of degree $m$ satisfying
$$f(xa)=f(x)a\text{ for all elements }a\in A,x\in M.$$
For left modules this becomes
$$f(ax)=(-1)^{mn}af(x)\text{ for all elements }a\in A^n,x\in M.$$

A morphism $f\colon M\to N$ between differential graded modules is {\em null-homotopic} if there is a morphism of the underlying graded modules $h\colon M^\natural\to N^\natural$ of degree $-1$ such that
$$f=d\circ h+h\circ d.$$

The {\em homotopy category} of differential graded $A$-modules denoted by $\mathbf H(A)$ is the quotient of $\mathbf C(A)$ with respect to the ideal of all null-homotopic morphisms. The homotopy category has the structure of a triangulated category (see \cite{V} or \cite{N} for the definition and properties of a triangulated category). Given a morphism $f\colon M\to N$ of differential graded modules the usual mapping cone construction gives a differential graded module $C(f)$ that may be visualised as follows.
$$\xymatrix{
\cdots\ar[r]^{-d_M^{i-2}}\ar[rd]_<<<{f^{i-2}} & M^{i-1} \ar[r]^{-d_M^{i-1}}\ar[rd]_<<<{f^{i-1}} & M^i \ar[r]^{-d_M^i}\ar[rd]_<<<{f^i} & M^{i+1} \ar[r]^{-d_M^{i+1}}\ar[rd]_<<<{f^{i+1}} & M^{i+2} \ar[r]^{-d_M^{i+2}} \ar[rd]_<<<{f^{i+2}}& \cdots \\
\cdots\ar[r]^<<<<{d_N^{i-3}} & N^{i-2}\ar[r]^<<<<{d_N^{i-2}} & N^{i-1} \ar[r]^<<<<{d_N^{i-1}} & N^i \ar[r]^<<<<<<{d_N^i} & N^{i+1} \ar[r]^<<<{d_N^{i+1}} & \cdots 
}$$
The mapping cone $C(f)$ fits into an exact triangle
$$M\overset f\to N\to C(f)\to\Sigma M$$
in the homotopy category, and the class of triangles isomorphic to these give all exact triangles of $\mathbf H(A)$. Following \cite{H} and \cite{Ke} one could also view $\mathbf H(A)$ as the stable category of the Frobenius category $\mathbf C(A)$, where the exact structure on $\mathbf C(A)$ is given by {\em semi-split} exact sequences, i.e.\ exact sequences of differential graded $A$-modules that split as sequences of the underlying graded modules.

A morphism $f\colon M\to N$ between differential graded modules is a {\em quasi-isomorphism} if it induces an isomorphism in cohomology.

The derived category $\mathbf D(A)$ is the localisation of $\mathbf H(A)$ with respect to the class of all quasi-isomorphisms. The derived category has a triangulated structure that is induced by the triangulated structure of $\mathbf H(A)$. The canonical localisation functor $\mathbf H(A)\overset{\rm can}\to\mathbf D(A)$ is an exact functor that sends all quasi-isomorphisms to isomorphism and is universal with with property.


We will denote the full subcategories of acyclic objects by superscript $ac$, the full subcategories of objects with degreewise finite dimensional cohomology by superscript $f$, the full subcategories of objects with bounded cohomology by superscript $b$, and the full subcategory of compact objects by superscript $c$, e.g.\ $\D^c(A)$ for the full subcategory of compact objects in the derived category.

%

Let $A,B$ be differential graded $k$-algebras. The {\em tensor product differential graded algebra} $A\otimes_kB$ is the complex that consists in degree $i$ of the $k$-vectorspace
$$\coprod_{p+q=i}A^p\otimes_k B^q,$$ and differential that is in degree $i$ given by 
$$d^i(a\otimes b):=d(a)\otimes b+(-1)^i a\otimes d(b)\text{ for all elements } a\in A^p,b\in B^q, p+q=i.$$
Together with the multiplication
$$(a\otimes b)(a'\otimes b'):=(-1)^{mn}aa'\otimes bb'\text{ for all elements } a\in A,a'\in A^m,b\in B^n,b'\in B$$
this gives a differential graded algebra.

An {\em $(A,B)$-bimodule} $M$ is by definition a differential graded $A\otimes_k B$-module. Equivalently, this is simultaneously a differential graded $A$ and $B$-module $M$ such that
$$(xb)a=(-1)^{mn}(xa)b\text{ for all elements } a\in A^m,b\in B^n,x\in M.$$
If we write for an $(A^{\op},B)$-bimodule the $A^\op$-module structure as $A$-module structure from the left, then the formula becomes
$$a(xb)=(ax)b\text{ for all elements } a\in A,b\in B,x\in M.$$
Note that an $A$-module is always an $(A,A^c)$-bimodule, where
$A^c$
is the {\em graded centre} of $A$, i.e.\ the homogeneous elements $a\in A^c$ are those satisfying $ab-(-1)^{|a||b|}ba=0$ for all homogeneous elements $b\in A$.

\subsection{Functors} 

Let $X$ be a complex of $k$-vectorspaces
$$X\colon~\cdots\longrightarrow X^{i-1}\overset {d^{i-1}}\longrightarrow X^{i}\overset{d^i}\longrightarrow X^{i+1}\longrightarrow\cdots.$$
The {\em $i$-th cohomology} $\Ho^i X$ of $X$ is the factor space $\Ker d^i/{\Im d^{i-1}}$. We denote the $\mathbb Z$-graded object $\coprod_{i\in\mathbb Z}\Ho^iX$ by $\Ho^*X$ and consider the {\em cohomology functors}
$$\Ho^*\colon\DGAlg\to\GAlg~\text{ and }~ \Ho^*\colon\mathbf C(A)\to\mathbf C_{\mathbf gr}(\Ho^*A).$$
Since $\Ho^*$ sends null-homotopies to zero and by definition quasi-isomorphisms to isomorphisms, the latter cohomology functor extends to $\mathbf H(A)$ and $\D(A)$. The functors $\Ho^*\colon\mathbf H(A)\to\mathbf C_{\mathbf gr}(\Ho^*A)$ and $\Ho^*\colon\D(A)\to\mathbf C_{\mathbf gr}(\Ho^*A)$ are {\em cohomological}, i.e. they send exact triangles to exact sequences.

Recall that an {\em exact} functor $\T\to \U$ between triangulated categories $\T,\U$ with suspension functors $\Sigma_\T,\Sigma_\U$ is a pair consisting of an additive functor $F$ and a natural isomorphism $\eta\colon F\circ\Sigma_\T\to\Sigma_\U\circ F$ such that every exact triangle $X\overset\a\to Y\overset\b\to Z\overset\g\to\Sigma X$ in $\T$ gives an exact triangle in $\U$
$$\xymatrix@=10pt{
FX\ar[rr]^{F\a} && FY \ar[rr]^{F\b} && FZ \ar@{-->}[rr]^{\eta_X\circ F\g}\ar[rd]_{F\g} && \Sigma (FX)\\
&&&&& F(\Sigma X)\ar[ru]_{\eta_X}
}.$$
If we talk about functors between triangulated categories, then we mean exact functors unless otherwise stated.

The {\em shift} autoequivalence
$$\Sigma\colon\mathbf C(A)\to\mathbf C(A)$$
assigns to each differential graded object $X$ its {\em shift by 1}. This means that the underlying complex is shifted by $1$, i.e.\ 
$$(\Sigma X)^i=X^{i+1},~d_{\Sigma X}=-d_X,$$
and the action of $A$ on $\Sigma M$ is as before on $M$.
For left $A$-modules the action of $A$ on the shifted module $\Sigma M$ gets a sign, namely
$$~ax:=(-1)^{n}ax,\text{ for all elements }a\in A^n,x\in \Sigma M.$$ 
Since $\Sigma$ respects homotopies and quasi-isomorphisms, $\Sigma$ induces shift functors on $\mathbf H(A)$ and $\mathbf\D(A)$ and coincides on these categories with the suspension functor of the triangulated structure.

Given a differential graded $A$-module $M$. Let $DM=\Hom_k(M,k)$ be the {\em $k$-dual} of $M$, i.e.\ the differential graded $A^\op$-module that is given by
$$(DM)^n=\Hom_k(M^{-n},k),\quad d^n_{DM}=(-1)^{n+1}\Hom_k(d_M^{-(n+1)},k),$$
and multiplication
$$(a\cdot f)(x)=(-1)^{ln+lm}f(xa)\text{ for all elements } a\in A^l, x\in M^m, f\in (DM)^n.$$
If we start with a left module, then we get an action from the right via
$$(f\cdot a)(x)=f(ax)\text{ for all elements } a\in A,f\in DM, x\in M.$$
Hence we have the {\em $k$-duality functor} 
$$D\colon\mathbf C(A)^\op\to\mathbf C(A^\op)$$ that preserves null-homotopic morphisms and quasi-isomorphisms. Therefore, it induces exact functors on the level of the corresponding homotopy categories and derived categories that will also be denoted by $D$. There are isomorphisms in on the level of the categories of differential graded modules, homotopy categories and derived categories $$\Hom(X,DY)\cong\Hom(Y,DX)$$ that are natural in $X$ and $Y$ (see Lemma~\ref{lem:kduality}). The $k$-duality functor $D$ restricts to a contravariant functor $\D^f(A)\to\D^f(A^\op)$. In this case $D$ is a duality. Also note that $DA$ is an $(A^\op,A)$-bimodule.

Let $M,N$ be differential graded $A$-modules. The {\em total-$\Hom$} of $M$ and $N$ is the $\mathbb Z$-graded object
$\H om_A(M,N)$ that consists in degree $i$ of the set of $A$-linear homomorphism $\Sigma^{-i}M^\natural\to N^\natural$ of the underlying graded modules. 
Together with the differential that is in degree $i$ given by
$$d^i(f):=d_N\circ f-f\circ d_{\Sigma^{-i}M}=d_N\circ f-(-1)^if\circ d_M.$$
this becomes a differential graded $k$-module. We get a bifunctor
$$\H om_A(-,-)\colon\mathbf C(A)^\op\times\mathbf C(A)\to\mathbf C(k)$$
that preserves homotopies, so we get an induced bifunctor on the homotopy category
$$\H om_A(-,-)\colon\mathbf H(A)^\op\times\mathbf H(A)\to\mathbf H(k).$$
Observe that taking cycles and accordingly cohomology of $\H om_A(M,N)$ in degree $i$ gives
$$\rm Z^i\H om_A(M,N)=\Hom_{\mathbf C(A)}(M,\Sigma^iN)$$
$$\Ho^i\H om_A(M,N)=\Hom_{\mathbf H(A)}(M,\Sigma^iN).$$

\begin{lem}\label{lem:kduality}
Let $X\in\mathbf C(A),Y\in\mathbf C(A^\op)$. Then there is an isomorphism in $\mathbf C(k)$ that is natural in $X$ and $Y$
$$\H om_A(X,DY)\cong\H om_{A^\op}(Y,DX).$$
In particular, 
$$\Hom_{\mathbf C(A)}(X,DY)\cong\Hom_{\mathbf C(A^\op)}(Y,DX),$$
$$\Hom_{\mathbf H(A)}(X,DY)\cong\Hom_{\mathbf H(A^\op)}(Y,DX).$$
\end{lem}

\begin{rem}
When we define the right derived $\Hom$-functor $\RHom(-,-)$ at the end of this section, then from the previous lemma we will also get that
$$\RHom_A(X,DY)\cong\RHom_{A^\op}(Y,DX),$$
$$\Hom_{\D(A)}(X,DY)\cong\Hom_{\D(A^\op)}(Y,DX).$$
Furthermore, we will see that there is a natural isomorphism of functors 
$$D\cong\RHom_A(-,DA)\colon\mathbf D(A)^\op\to\mathbf D(A^\op).$$
\end{rem}

Let $M$ be an $A$-module and $N$ be a left $A$-module. The {\em total-tensor-product} $M\otimes_A N$ of $M$ and $N$ is the $\mathbb Z$-graded object that consists in degree $n$ of the $k$-vectorspace
$${(\coprod_{p+q=n}M^p\otimes_k N^q)}/\X$$
where $\X$ is the subspace generated by elements $$xa\otimes y-x\otimes ay,\text{ where }x\in M^p,y\in N^q,a\in A^r,p+q+r=n.$$
Together with the differential
$$d^i(x\otimes y):=dx\otimes y+(-1)^m x\otimes dy\text{ for all elements }x\in M^m,y\in N$$ 
this gives a differential graded $k$-module. We get a bifunctor
$$-\otimes_A-\colon\mathbf C(A)\times\mathbf C(A^\op)\to\mathbf C(k)$$
that preserves homotopies, so we get an induced bifunctor on the homotopy categories
$$-\otimes_A-\colon\mathbf H(A)\times\mathbf H(A^\op)\to\mathbf H(k).$$

Let $A,B$ be differential graded algebras, further let $M$ be an $(A^\op,B)$-bimodule, $X$ an $A$-module, and $Y$ a $B$-module. Then $X\otimes_AM$ has a $B$-module structure via
$$(x\otimes m)b=x\otimes mb\text{ for all elements }x\in X,m\in M,b\in B,$$
and we get a functor
$$-\otimes_AM\colon\mathbf C(A)\to\mathbf C(B)\text{ (respectively an exact functor }\mathbf H(A)\to\mathbf H(B)).$$
Furthermore, $\H om_B(M,Y)$ has an $A$-module structure via
$$(fa)(m)=f(am)\text{ for all elements }f\in\H om_A(M,Y),a\in A,m\in M,$$
and we get a functor
$$\H om_B(M,-)\colon\mathbf C(B)\to\mathbf C(A)\text{ (respectively an exact functor }\mathbf H(B)\to\mathbf H(A)).$$
Moreover, $\H om_B(M,-)$ is right adjoint to $-\otimes_AM$.

To define `derived' functors on the level of derived categories one needs for differential graded modules an analogue of projective (respectively injective) resolutions. A differential graded $A$-module $M$ is {\em homotopically projective} if $\H om_A(M,-)$ preserves quasi-isomorphisms. Let $\mathbf C_p(A)$ respectively $\mathbf H_p(A)$ denote the category respectively homotopy category of all homotopically projective differential graded modules.
A differential graded $A$-module $M$ is {\em homotopically injective} if $\H om_A(-,M)$ preserves quasi-isomorphisms. Let $\mathbf C_i(A)$ respectively $\mathbf H_i(A)$ denote the category respectively homotopy category of all homotopically injective differential graded modules.

\begin{lem}\label{lem:homproj} Let $M$ be a differential graded $A$-module. The following statements are equivalent
\begin{enumerate}
\item $M$ is homotopically projective.
\item[(1')] $\Hom_{\mathbf H(A)}(M,-)$ sends quasi-isomorphisms to isomorphisms, i.o.w. $M$ has in $\mathbf H(A)$ the following unique lifting property
$$\xymatrix@=5mm{
& M \ar[d]\ar@{-->}[ld]_{\exists_1} \\
X \ar[r]_{qi} & Y
}$$
\item $\H om_A(M,-)$ preserves acyclicity. (Spaltenstein \cite{S} calls $M$ {\em K-projective})
\item[(2')] $\Hom_{\mathbf H(A)}(M,\mathbf H_{\rm ac}(A))=0$
\item $M\in\langle A\rangle_{\rm loc}\subseteq\mathbf H(A)$.
\item $M$ is homotopy equivalent to a semi-free module, i.e.\ a module that has an exhaustive ascending filtration in $\mathbf C(A)$ with subquotients isomorphic to direct sums of copies of shifts of $A$.
\item $M$ is homotopy equivalent to a module that has an exhaustive ascending filtration in $\mathbf C(A)$ with subquotients isomorphic to direct summands of direct sums of copies of shifts of $A$ and the consecutive inclusions split as inclusions of graded modules. (Keller \cite{Ke} says that such a module has {\em property (P)})
\end{enumerate}
\end{lem}
\begin{proof}
The equivalence of (1) and (1') as well as (2) and (2') follows from the natural isomorphism $\Ho^i\H om_A(M,-)\cong\Hom_{\mathbf H(A)}(M,\Sigma^i-)$. The equivalence of (1') and (2') holds since in a triangle $X\overset f\to Y\to Z\to\Sigma X$ the morphism $f$ is a quasi-isomorphism if and only if $Z$ is acyclic, and $\Hom_{\mathbf H(A)}(M,-)$ is a homological functor.\\
$(4)\Rightarrow(5)$: This is clear.\\
$(5)\Rightarrow(3)$: If $M$ is a module with filtration $0\subseteq M(0)\subseteq M(1)\subseteq\cdots\subseteq M$ that has the properties from (5), then $M$ is the homotopy colimit in $\mathbf H(A)$ of the objects $M(i)$ that appear in the semi-free filtration. More precisely, there is a triangle
$$\coprod_i M(i)\to\coprod_j M(j)\to M\to\Sigma\coprod_i M(i)$$
built from the maps ${1\brack{-\text{incl}}}\colon M(i)\to M(i)\amalg M(i+1)$.
The $M(i)$ are in $\langle A\rangle_{\rm loc}\subseteq\mathbf H(A)$ since the exact sequences
$$0\to M(i)\to M(i+1)\to M(i+1)/M(i)\to 0$$
split as sequences of graded modules. Hence also the homotopy colimit $M$ and all modules that are homotopy equivalent to $M$ belong to $\langle A\rangle_{\rm loc}\subseteq\mathbf H(A)$.\\
$(3)\Rightarrow\text{(2')}$: Let $\M$ be the class of objects $M\in\mathbf H(A)$ that have the property that $\Hom_{\mathbf H(A)}(M,\mathbf H_{\rm ac}(A))=0$. It is a triangulated subcategory that contains $A$ and is closed under coproducts. Hence $\M$ contains $\langle A\rangle_{\rm loc}$.\\
$\text{(1')}\Rightarrow \text{(4)}$: Let $M$ be an object that has the lifting property from (1'). Consider a semi-free resolution $a\colon L\to M$. The lifting property for $M$ gives a morphism $b\colon M\to L$ with $ab=\id_M$ in $\mathbf H(A)$.
$$\xymatrix@=5mm{
& M \ar@{=}[d]\ar@{-->}[ld]_b \\
L \ar[r]_a & M
}$$
The already shown implication $(4)\Rightarrow\text{(1')}$ gives that also $L$ has the lifting property. Since $b$ is a quasi-isomorphism we also get $ba=\id_M$ in $\mathbf H(A)$.
$$\xymatrix@=5mm{
& L \ar@{=}[d]\ar@{-->}[ld]_a \\
M \ar[r]_b & L
}$$
Hence $M$ is homotopy equivalent to its semi-free resolution $L$.\\
\end{proof}

Note that a dual version of this lemma holds for homotopically injective differential graded modules.

Using the Brown representability theorem (cf.\ \cite{K},\cite{K2}, or \cite{N}) one sees that the inclusions $\mathbf H_p(A)\to\mathbf H(A)$ respectively $\mathbf H_i(A)\to\mathbf H(A)$ have a right adjoint $p\colon\mathbf H(A)\to\mathbf H_p(A)$ respectively a left adjoint $i\colon\mathbf H(A)\to\mathbf H_i(A)$.
The compositions
$$\mathbf H_p(A)\overset{\rm {incl}}\to\mathbf H(A)\overset{\rm {can}}\to\D(A)$$
$$\mathbf H_i(A)\overset{\rm {incl}}\to\mathbf H(A)\overset{\rm {can}}\to\D(A)$$
are equivalences of triangulated categories. The quasi-inverses that are induced by $i$ and $p$ exist by the universal property of the localisation functor and will also be denoted by $i$ and $p$.
$$\xymatrix@=5mm{
\mathbf H(A) \ar[r]\ar[d]_{p}& \D(A) \ar@{-->}[ld]^p && \mathbf H(A) \ar[r]\ar[d]_{i}& \D(A) \ar@{-->}[ld]^i\\
\mathbf H_p(A)  & && \mathbf H_i(A)
}$$
Given differential graded modules $X,Y$ there are natural isomorphisms 
$$\Hom_{\D(A)}(X,Y)\cong\Hom_{\D(A)}(pX,Y)\cong\Hom_{\mathbf H(A)}(pX,Y).$$

For a given exact functor $F\colon\mathbf H(A)\to\mathbf H(B)$ we define the {\em left derived functor} $LF$ respectively {\em right derived functor} $RF$ of $F$ to be the composition
$$LF\colon\D(A)\underset\simeq{\overset p\to}\mathbf H_p(A)\overset{incl}\to\mathbf H(A)\overset F\to\mathbf H(B)\overset{can}\to\D(B)$$
$$RF\colon\D(A)\underset\simeq{\overset i\to}\mathbf H_i(A)\overset{incl}\to\mathbf H(A)\overset F\to\mathbf H(B)\overset{can}\to\D(B)$$

\begin{rem}\label{rem:adjoint}
If $F\colon\mathbf H(A)\to\mathbf H(B)$ is right adjoint to $G\colon\mathbf H(B)\to\mathbf H(A)$, then the right derived functor $RF\colon\D(A)\to\D(B)$ is right adjoint to the left derived functor $LG\colon\D(B)\to\D(A)$.
\end{rem}

\begin{rem} 
\begin{enumerate}
\item The right derived $\Hom$-functor $\RHom_A(-,-):=\H om_A(-,i-)$ is a bifunctor
$$\RHom_A(-,-)\colon\D(A)^\op\times\D(A)\to\D(k)$$ and naturally isomorphic to the left derived one $\LHom_A(-,-):=\H om_A(p-,-)$.
\item It holds $\Ho^i\RHom_A(M,N)\cong\Hom_{\D(A)}(M,\Sigma^i N)$.
\item The left derived tensor-functor $-\otimes_A^L-:=(p-)\otimes_A-$ is a bifunctor
$$-\otimes^L_A-\colon\D(A)\times\D(A^\op)\to\D(k).$$
\item Let $A,B$ be differential graded algebras and $M$ be an $(A^\op,B)$-bimodule then we have exact functors $$\RHom_B(M,-)\colon\D(B)\to\D(A),$$
$$-\otimes_A^LM\colon\D(A)\to\D(B).$$
Moreover, $\RHom_B(M,-)$ is right adjoint to $-\otimes_A^LM$.
\end{enumerate}
\end{rem}

%

\begin{lem}\label{lem:equ}
Let $f\colon A\to B$ be a quasi-isomorphism of differential graded $k$-algebras. Consider the adjoint pair of functors $$G:=-\otimes_A^L B\colon\D(A)\to\D(B)\text{ and } F:=\RHom_B(_AB,-)\colon\D(B)\to\D(A).$$ Then
\begin{enumerate}
 \item The functors $F$ and $G$ are mutually inverse equivalences of triangulated categories.
\item Let $G'$ be the functor $B\otimes_A^L -\colon\D(A^\op)\to\D(B^\op)$. For all $M\in\D(A),~X\in\D(A^\op)$ there is an isomorphism of graded $k$-vectorspaces $$\Ho^*(M\otimes_A^L X)\cong\Ho^*(GM\otimes_B^L G'X).$$
In particular, for all $M\in\D(A)$ there is an isomorphism of graded $k$-vectorspaces $$\Ho^*M\cong \Ho^*GM.$$
For the functors $F$ and $F':=\RHom_B(B_A,-)\colon\D(B^\op)\to\D(A^\op)$ the analogous statements hold.
\end{enumerate}
\end{lem}
\begin{proof}
(1) The functor $-\otimes_A^L B$ is exact and preserves coproducts. Furthermore, $-\otimes_A^L B$ maps the compact generator $A$ of $\D(A)$ to the compact generator $B$ of $\D(B)$. The following commutative diagram
$$\xymatrix{
\Hom_{\D(A)}(A,\Sigma^n A) \ar[r]\ar[d]^\cong & \Hom_{\D(B)}(A\otimes_A^L B,\Sigma^n A\otimes_A^L B) \ar[r]^{\quad\quad\cong} & \Hom_{\D(B)}(B,\Sigma^n B)\ar[d]^\cong \\
\Ho^n A \ar[rr]_{\Ho^n f}^\cong && \Ho^n B
}$$
gives isomorphisms $\Hom_{\D(A)}(A,\Sigma^n A) \to\Hom_{\D(B)}(A\otimes_A^L B,\Sigma^n A\otimes_A^L B)$ for all $n\in\mathbb Z$. Using d\'evissage arguments one gets that $-\otimes_A^L B\colon\D(A)\to\D(B)$ is fully faithful and dense.\\
(2) Note that $B$ considered as an $(A^\op,A)$-bimodule is quasi-isomorphic to $A$. We calculate in $\D(k)$
\begin{multline*}
GM\otimes_B^L G'X\cong M\otimes_A^L B\otimes_B^L B\otimes_A^L X\cong M\otimes_A^L B\otimes_A^L X\cong M\otimes_A^L A\otimes_A^L X\cong M\otimes_A^L X.
\end{multline*}
Taking cohomology gives the desired result. The analogous results for the functors $F$ and $F'$ follow immediately from (1).
\end{proof}

We call an object $X\in\D(A)$ {\em generator} for $\D(A)$ if $\langle X\rangle_{\rm loc}=\D(A)$. Further recall that an object $X$ is {\em compact} if $\Hom_{\D(A)}(X,-)$ preserves coproducts, i.e.\ the natural morphism $$\coprod_\a\Hom_{\D(A)}(X,X_\a)\to\Hom_{\D(A)}(X,\coprod_\a X_\a)$$
is an isomorphism for all families of objects $\{X_\a\}$, or equivalently every morphism from $X$ to $\coprod_\a X_\a$ factors through the coproduct of some finite subfamily $\{X_{\a_1},\dots,X_{\a_n}\}$ of $\{X_\a\}$. Neeman \cite{N2} has shown that in the derived category the compact objects are those finitely build from  $A$, i.e.\ $\D^c(A)=\langle A\rangle_{\rm thick}$ (see also \cite{Ke}). This immediately gives:

\begin{lem}
An object $X\in\D(A)$ is compact if and only if $\RHom_A(X,-)$ preserves coproducts.
\end{lem}

\subsection{Resolutions and minimality}

Minimal semi-free resolutions are used by J\o rgensen in his work on simply connected differential graded algebras. The existence of those for all differential graded modules over simply connected differential graded algebras is shown in \cite{AFH2}. For an arbitrary differential graded algebra $A$ one has the following lemma (see \cite[Theorem~7.3.2]{AFH} or \cite[Proposition~6.6]{FHT})

\begin{lem}
Every differential graded $A$-module $M$ has a strict semi-free resolution, i.e.\ there exists a semi-free module $L$ and a surjective morphism of differential graded modules $L\to M$ that is a quasi isomorphism.
\end{lem}

A differential graded module $L$ is called {\em minimal} if every morphism of differential graded modules $L\to L$ that is a homotopy equivalence is already an isomorphism. Recall also that a semi-free differential graded module is in particular homotopically projective (see Lemma~\ref{lem:homproj}). The uniqueness of minimal homotopically projective resolutions is given by

\begin{lem}
A minimal homotopically projective resolution is unique up to isomorphism, i.e.\ given two such resolutions $f\colon L\to M$ and $f'\colon L'\to M$ there exists an isomorphism of differential graded modules $a\colon L\to L'$ such that $f=f'\circ a$.
$$\xymatrix{
L \ar@{-->}[dd]_a\ar[rd]^f \\
& M\\
L' \ar[ru]_{f'} 
}$$
\end{lem}
\begin{proof}
Since $L$ and $L'$ have the lifting property (1') from Lemma~\ref{lem:homproj} we get unique morphisms $a\colon L\to L'$ and $b\colon L'\to L$ in $\mathbf H(A)$ such that the following diagram commutes in $\mathbf H(A)$.
$$\xymatrix{
L \ar@{-->}[d]_a\ar[rd]^f \\
L' \ar[r]^{f'}\ar@{-->}[d]_b & M\\
L \ar[ru]_f
}$$
The uniqueness of the lifting morphisms gives that $ba$ is homotopic to $\id_L$ and similarly $ab$ homotopic to $\id_{L'}$. Since $L$ and $L'$ are minimal it follows that $a$ and $b$ are isomorphisms in $\mathbf C(A)$.
\end{proof}
We also note the following lemma on the indecomposability of minimal homotopically projective differential graded modules.
\begin{lem}
Let $L$ be a minimal homotopically projective differential graded $A$-module. If $L$ is indecomposable in $\D(A)$, then $L$ is also indecomposable in $\mathbf C(A)$.
\end{lem}
\begin{proof}
Consider
$$\End_{\mathbf C(A)}(L)\overset\pi\to\End_{\mathbf H(A)}(L)\cong\End_{\D(A)}(L)$$
Take an idempotent $e$ in $\End_{\mathbf C(A)}(L)$. The image $\pi(e)$ is an idempotent in $\End_{\mathbf H(A)}(L)$. Since $\End_{\mathbf H(A)}(L)\cong\End_{\D(A)}(L)$ and $L$ is indecomposable in $\D(A)$ we have that $\pi(e)$ is a trivial idempotent. Without loss of generality let $\pi(e)=\id_L$. In particular $e$ is a homotopy equivalence. Since $L$ is minimal it follows that $e$ is an isomorphism of differential graded modules. Hence $e=\id_L$ and $L$ is indecomposable in $\mathbf C(A)$.
\end{proof}


\newpage

\section{Auslander-Reiten theory}

In this appendix we summarise for the reader that is not familiar with Auslander-Reiten theory some of the basic definitions and concepts of Auslander-Reiten theory.

\subsection{The Auslander-Reiten quiver}

Auslander-Reiten sequences or almost split sequences have been introduced by Auslander and Reiten in \cite{A1}. They have shown that the category of finitely presented modules $\mod R$ over an Artin-algebra $R$ has Auslander-Reiten sequences. The corresponding Auslander-Reiten quiver gives a `nice' visualisation of the module category, and a lot of information about the module category is contained in this combinatorial invariant. For example, in the case of a finite dimensional algebra of finite representation type over an algebraically closed field of characteristic different from $2$ it is possible to reconstruct the algebra from its Auslander-Reiten quiver. Auslander-Reiten theory has been of substantial interest in the representation theory of finite dimensional algebras for more than 20 years. Happel has defined an analogue for triangulated categories in \cite{H1} (see also in his book \cite{H}) and shown that the bounded derived category $\D^b(\mod A)$ of finite dimensional modules over a finite dimensional algebra $A$ has Auslander-Reiten triangles if and only if $A$ has finite global dimension (see \cite{H1} and \cite{H2}). Here we summarise some foundations in the triangulated situation.

\begin{defn}
Let $\T$ be a triangulated category. An {\em Auslander-Reiten triangle} in $\T$ is a triangle
$$X\overset f\to Y\overset g\to Z\overset h\to\Sigma X$$
with the following properties.
\begin{enumerate}
\item every map $X\to Y'$ that is not a split monomorphism factors through f,
\item every map $Y'\to Z$ that is not a split epimorphism factors through g,
\item $h$ is not the zero map.
\end{enumerate}
\end{defn}

Note that the end terms $X$ and $Z$ of an Auslander Reiten triangle
$$X\to Y\to Z\to\Sigma X$$
are indecomposable with local endomorphism rings and each end term determines the Auslander-Reiten triangle up to isomorphism. Therefore we can define for those $Z$ where an Auslander-Reiten triangle ends the {\em Auslander-Reiten translate} $\tau$ to be $\tau(Z)=X$.

We say that the category $\T$ has left (respectively right) Auslander-Reiten triangles if for every indecomposable object $X$ (respectively $Z$) there is an Auslander-Reiten triangle
$$X\overset f\to Y\overset g\to Z\overset h\to\Sigma X.$$
Further we say that the category $\T$ has Auslander-Reiten triangles if it has left and right Auslander-Reiten triangles.

The {\em Auslander-Reiten quiver} $\Gamma(\C)$ of an additive category $\C$ is the quiver consisting of all isomorphism classes $[X]$ of indecomposable objects $X$ as vertices and has an arrow starting in a vertex $[X]$ and ending in a vertex $[Y]$ if there exists an {\em irreducible} morphism $X\overset f\to Y$, i.e.\ $f$ is neither a split monomorphism nor a split epimorphism and in every factorisation
$$\xymatrix{
X\ar[rr]^f \ar[rd]_g && Y\\
&Z\ar[ru]_h
}$$
$g$ is a split monomorphism or $h$ a split epimorphism.

The connection to Auslander-Reiten triangles is given by the following Lemma.

\begin{lem}\cite[Proposition~4.3]{H}
Let $\T$ be a triangulated Krull-Remak-Schmidt category and $X\overset f\to Y\overset g\to Z\overset h\to\Sigma X$ an Auslander-Reiten triangle in $\T$. Then
\begin{enumerate}
\item The morphisms $f$ and $g$ are irreducible.
\item If $f_1\colon Y_1\to Z$ is an irreducible morphism, then there is a split monomorphism $\a\colon Y_1\to Y$ such that $f_1=g\circ\a$.
\item If $g_1\colon X\to Y_1$ is an irreducible morphism, then there is a split epimorphism $\b\colon Y\to Y_1$ such that $g_1=\b\circ f$.
\end{enumerate}
\end{lem}

If $\T$ is a triangulated Krull-Remak-Schmidt category that has Auslander-Reiten triangles, then by the lemma all the information for the Auslander-Reiten quiver is contained in the Auslander-Reiten triangles. In this case the Auslander-Reiten quiver also has some extra structure: The Auslander-Reiten translate $\tau$ is defined on all isomorphism classes and becomes a quiver automorphism satisfying the property that the set of predecessors of every vertex $[Z]$ equals the set of all successors of $\tau([Z])$, i.e.\ $\Gamma(\T)$ has the structure of a {\em stable translation quiver}. Furthermore, the multiplicity $a_{MN}$ of $M$ as direct summand in the middle term of an Auslander-Reiten triangle ending in $N$ and the multiplicity $a'_{MN}$ of $N$ as direct summand in the middle term of an Auslander-Reiten triangle starting in $M$ give a {\em valuation} $(a_{MN},a'_{MN})$ of an arrow $[M]\to[N]$. This valuation satisfies $(a_{\tau NM},a'_{\tau NM})=(a'_{MN},a_{MN})$ and the Auslander-Reiten quiver has the structure of a {\em valued stable translation quiver}.

A {\em (sub)additive function} on a valued stable translation quiver is a function $f$ that assigns to each vertex a natural number such that 
$$f(Z)+f(\tau Z)\underset{(\geq)}=\sum_{Y\in Z^-} a_{YZ}f(Y),$$
where $Z^-$ denotes the set of successors of $Z$. Combinatorial methods give results on the structure of valued stable translation quivers depending on the existence of (un)bounded (sub)additive functions (see \cite{HPR}).

Given an oriented tree $\Delta$ with valuation $\n$ let $\mathbb Z\Delta$ be the valued stable translation quiver defined as follows: The set of vertices of $\mathbb Z\Delta$ is the Cartesian product of $\mathbb Z$ with the set of vertices of $\Delta$. For every arrow $X\to Y$ in $\Delta$ let 
$$(n,X)\to(n,Y)\text{ and }(n,Y)\to(n-1,X),~n\in\mathbb Z$$ be arrows in $\mathbb Z\Delta$. Define the translation in $\mathbb Z\Delta$ by $\tau(n,X)=(n+1,X)$ and the valuation by $\n((n,X)\to(n,Y))=\n(X\to Y)$ and $\n((n,X)\to(n-1,Y))=\sigma\n(X\to Y)$ where $\sigma$ permutes the first and second coordinate. Note that up to isomorphism $\mathbb Z\Delta$ does not depend on the orientation of $\Delta$.

\begin{exm}
Consider the oriented tree
$$\overrightarrow{A_3}:\xymatrix{
1 \ar[r] & 2 \ar[r] & 3 
}$$
with trivial valuation $(1,1)$ for all arrows.
Then $\mathbb Z(\overrightarrow{A_3})$ is the following stable translation quiver
$$\xymatrix@=5mm{
\cdots &(2,3) \ar[rd] && (1,3)\ar[rd]\ar@{-->}_\tau[ll] && (0,3)\ar[rd]\ar@{-->}_\tau[ll] && (-1,3)\ar@{-->}_\tau[ll] & \cdots \\
\cdots && (1,2) \ar[ru]\ar[rd]\ar@{--}[l] && (0,2) \ar[ru]\ar[rd]\ar@{-->}_\tau[ll] && (-1,2) \ar[ru]\ar[rd]\ar@{-->}_\tau[ll] & \ar@{-->}[l] & \cdots \\
\cdots &(1,1) \ar[ru] && (0,1) \ar[ru]\ar@{-->}_\tau[ll] && (-1,1) \ar[ru]\ar@{-->}_\tau[ll] && (-2,1) \ar@{-->}_\tau[ll] & \cdots
}$$
with also trivial valuation $(1,1)$ for all arrows. This quiver coincides with the Auslander-Reiten quiver of the the bounded derived category $\D^b(\mod k\overrightarrow{A_3})$ of finitely presented modules over the path algebra $k\overrightarrow{A_3}$.
\end{exm}


\subsection{The functorial approach}

In \cite{A2} Auslander gives a functorial approach to Auslander-Reiten theory. Here we state some basic results in the triangulated setup.

Let $\T$ be a triangulated category
and let $\widehat\T:=\fp(\T^\op,\Ab)$ be the functor category of finitely presented additive functors $\T^\op\to\Ab$, where a functor $F$ is {\em finitely presented} (or {\em coherent}) if there exists an exact sequence $(-,Y)\to(-,X)\to F\to 0$. By Yoneda's lemma the natural transformations between two coherent functors form a set and $\widehat\T$ is an abelian Frobenius category called the {\em abelianisation} of $\T$. The Yoneda embedding 
$$\T\to\widehat\T,~X\mapsto \Hom_\T(-,X)$$
 is a cohomological functor. The projective (equivalently injective) objects in $\widehat\T$ are exactly the functors in the image of the Yoneda functor, i.e.\ the representable functors. Each indecomposable object $X$ in $\T$ with local endomorphism ring gives rise to a simple functor $$S_X:=\Hom_\T(-,X)/{\rad\Hom_\T(-,X)},$$ where the {\em radical} of a functor is defined as the intersection of all maximal subfunctors.

\begin{prop}
Let $\T$ be triangulated category and $A\to B\to C\to$ be a triangle in $\T$. Then the following are equivalent.
\begin{enumerate}
\item The triangle $$A\to B\to C\to\Sigma A$$ is an Auslander-Reiten triangle.
\item The object $C$ in $\T$ is indecomposable with local endomorphism ring and the sequence $$(-,B)\to(-,C)\to S_C\to 0$$ is a minimal projective presentation of $S_C$.
\end{enumerate}
\end{prop}
\begin{proof}
The assertion follows by imitating the proof of \cite[II~Pop~4.4]{A2} and using \cite[II.~Proposition~1.9, II.~ Corollary~1.11, II.~Proposition~2.3 and II.~Proposition~2.7]{A2}.
\end{proof}
\sloppy
\begin{prop}\label{prop:simplefunct}
Let $\T$ be a triangulated Krull-Remak-Schmidt category having Auslander-Reiten triangles. Then there is a one-to-one correspondence between the isomorphism classes of indecomposable objects in $\T$ and the isomorphism classes of simple objects in $\widehat\T$ given by the assignment
$$X\mapsto S_X:=\Hom_\T(-,X)/{\rad\Hom_\T(-,X)}.$$
In particular, there is a one-to-one correspondence between Auslander-Reiten triangles in $\T$ and minimal projective presentations of simple functors in $\widehat\T$
\end{prop}
\fussy
\begin{proof}
Given a simple functor $F\colon\T^\op\to\Ab$ there exists up to isomorphism a unique indecomposable object $X$ in $\T$ such that $F(X)\neq 0$. Choosing this $X$ gives a map from simple functors in $\widehat\T$ to indecomposable objects in $\T$. From Yoneda's Lemma we get a morphism of functors $\Hom_\T(-,X)\to F$. Since $F$ is simple this has to be an epimorphism. But the radical $\rad\Hom_\T(-,X)$ is the unique maximal subfunctor of $\Hom_\T(-,X)$. Hence we get $F\cong S_X$. Conversely, $S_X$ is finitely presented by the previous proposition and satisfies $S_X(X)\neq 0$. This gives the stated one-to-one correspondence. The second statement then follows immediately from the previous proposition.
\end{proof}

\newpage

\section{Some topological background}

In this appendix we summarise a few topological definitions and facts that appear in this thesis for the reader that is not familiar with algebraic topology. One finds further explanations in every classical book on algebraic topology, e.g.\ \cite{Spanier},\cite{Hat}. Furthermore, we mention some results from rational homotopy theory that are related to simply connected cochain differential graded algebras, the main objects of study in this thesis. As reference for this results on rational homotopy theory we give \cite{FHT}, \cite{AH} and \cite{Hess}.
 
\subsection{The singular cochain differential graded algebra}
The {\em singular chain complex} $C_*(X;k)$ of a topological space $X$ consists in degree $n$ of the free $k$-module with basis the set of singular $n$-simplices, i.e.\ continuous maps from the standard $n$-simplex $\Delta^n=\{\sum_{i=0}^n\la_ie_i\mid\la_i\geq 0, \sum\la_i=1\}\subseteq\mathbb R^{n+1}$ to $X$, and differential given by $\sum_{i=1}^n(-1)^i\delta_i$. Here $\delta_i$ is the {\em $i$-th face map} that assigns to an $n$-simplex $\sigma\colon\Delta^n\to X$ its composition $\sigma\circ\la_i\colon\Delta^{n-1}\to\Delta^n\to X$, where $\la_i\colon\Delta^{n-1}\to\Delta^n$ is the {\em $i$-th face inclusion} that is given by the $(n+1)\times n$-matrix
$$\left(\begin{smallmatrix}
E_i & 0  \\
0   & 0  \\
0   & E_{n-i}
\end{smallmatrix}\right).$$

The {\em singular cochain differential graded algebra} $C^*(X;k)$ is as a complex of $k$-modules the $k$-linear dual of the singular chain complex. In addition, the {\em cup product} defines a multiplicative structure on $C^*(X;k)$ that makes it into a differential graded algebra. The cup product can be defined by the formula
$$(f\cup g)(\sigma)=(-1)^{k(n-k)}f(\sigma\circ
\left(\begin{smallmatrix}
E_{k+1}  \\
0    \\
\end{smallmatrix}\right)
)
g(\sigma\circ
\left(\begin{smallmatrix}
0 \\
E_{n-k+1}  \\
\end{smallmatrix}\right)
)$$
where $f\in C^k(X), g\in C^{n-k}(X),$ and $\sigma\colon\Delta^n\to X$.
One calculates that the cup product satisfies the Leibniz rule
$$d(f\cup g)=d(f)\cup g+(-1)^kf\cup d(g).$$
As a consequence there is an induced cup product in cohomology. One can show that the cohomology algebra is graded commutative, i.e.\ $f\cup g=(-1)^{k(n-k)}g\cup f$.

Let $X$ and $Y$ be path connected spaces. Recall that a continuous map $f\colon X\to Y$ is a {\em weak homotopy equivalence} if the induced morphism between the homotopy groups $\pi_nf\colon\pi_nX\to\pi_nY$ is an isomorphism for all $n\in\mathbb N$. Two path connected spaces are {\em weakly homotopy equivalent} if they are equivalent in the induced equivalence relation. Further recall that a path connected topological space $X$ is {\em simply connected} if the fundamental group $\pi_1(X)$ is trivial.  In this case the Hurewicz theorem gives $\Ho_0(X;k)\cong k$ and $\Ho_1(X;k)=0$. Since $k$ is a field it follows that $\Ho^0(X;k)\cong k$ and $\Ho^1(X;k)=0$. Hence $C^*(X;k)$ is a simply connected cochain differential graded algebra.

\subsection{Rational homotopy theory}
Rational homotopy theory is the study of rational homotopy types of spaces, where a map between two simply connected spaces $\phi\colon X\to Y$ is a {\em rational homotopy equivalence} if $\pi_*(\phi)\otimes\mathbb Q$ is an isomorphism or equivalently $\Ho_*(\phi;\mathbb Q)$ (respectively $\Ho^*(\phi;\mathbb Q)$) is an isomorphism.
It is one result in rational homotopy theory that there is a bijection between the rational homotopy types of simply connected spaces of finite dimensional cohomology and weak equivalence classes of simply connected commutative cochain differential graded algebras of finite type over $\mathbb Q$. The bijection is given by a contravariant functor $A_{PL}$ that assigns to a simply connected space $X$ a commutative differential graded algebra $A_{PL}(X)$ that is weakly equivalent to the singular cochain differential graded algebra $C^*(X;\mathbb Q)$. Moreover, taking a minimal Sullivan model of a simply connected commutative cochain differential graded algebra of finite type gives a bijection between the weak equivalence classes of simply connected commutative cochain differential graded algebras of finite type and isomorphism classes of minimal Sullivan algebras over $\mathbb Q$. Minimal Sullivan algebras are in model theoretic terms cofibrant replacements in $\CDGAlg$, and they are useful for calculations. In particular, it is possible in characteristic zero to construct a commutative finite dimensional model for the singular cochain differential graded algebra $C^*(X;k)$ of a simply connected space $X$ (see \cite{FHT}).

A topological space is called {\em formal} if the commutative cochain differential graded algebra $A_{PL}(X)$ is weakly equivalent in $\CDGAlg$ to its cohomology algebra $\Ho^*A_{PL}(X)$. In particular, for a formal space $X$ the singular cochain differential graded algebra $C^*(X;\mathbb Q)$  is weakly equivalent in $\DGAlg$ to its cohomology algebra $\Ho^*(X;\mathbb Q)$. 



\newpage

\begin{thebibliography}{10}

\bibitem{AHi}
{\sc Adams, J.~F., and Hilton, P.~J.}
\newblock On the chain algebra of a loop space.
\newblock {\em Comment. Math. Helv. 30\/} (1956), 305--330.

\bibitem{A2}
{\sc Auslander, M.}
\newblock Functors and morphisms determined by objects.
\newblock In {\em Representation theory of algebras (Proc. Conf., Temple Univ.,
  Philadelphia, Pa., 1976)}. Dekker, New York, 1978, pp.~1--244. Lecture Notes
  in Pure Appl. Math., Vol. 37.

\bibitem{A1}
{\sc Auslander, M., and Reiten, I.}
\newblock Representation theory of {A}rtin algebras. {III}. {A}lmost split
  sequences.
\newblock {\em Comm. Algebra 3\/} (1975), 239--294.

\bibitem{A}
{\sc Auslander, M., and Reiten, I.}
\newblock Applications of contravariantly finite subcategories.
\newblock {\em Adv. Math. 86}, 1 (1991), 111--152.

\bibitem{AH}
{\sc Avramov, L., and Halperin, S.}
\newblock Through the looking glass: a dictionary between rational homotopy
  theory and local algebra.
\newblock In {\em Algebra, algebraic topology and their interactions
  (Stockholm, 1983)}, vol.~1183 of {\em Lecture Notes in Math.} Springer,
  Berlin, 1986, pp.~1--27.

\bibitem{ABYM}
{\sc Avramov, L.~L., Buchweitz, R.-O., Iyengar, S., and Miller, C.}
\newblock Homology of perfect complexes.
\newblock arXiv:math/0609008v2 [math.AC], 2006.

\bibitem{AF}
{\sc Avramov, L.~L., and Foxby, H.-B.}
\newblock Locally {G}orenstein homomorphisms.
\newblock {\em Amer. J. Math. 114}, 5 (1992), 1007--1047.

\bibitem{AFH}
{\sc Avramov, L.~L., Foxby, H.-B., and Halperin, S.}
\newblock Differential graded homological algebra.
\newblock In preparation, 2004.
\newblock Version from 02.07.2004.

\bibitem{AFH2}
{\sc Avramov, L.~L., Foxby, H.-B., and Halperin, S.}
\newblock Resolutions of differential graded modules.
\newblock In preparation, 2007.

\bibitem{DG}
{\sc Dwyer, W.~G., and Greenlees, J. P.~C.}
\newblock Complete modules and torsion modules.
\newblock {\em Amer. J. Math. 124}, 1 (2002), 199--220.

\bibitem{DGI}
{\sc Dwyer, W.~G., Greenlees, J. P.~C., and Iyengar, S.}
\newblock Duality in algebra and topology.
\newblock {\em Adv. Math. 200}, 2 (2006), 357--402.

\bibitem{FHT2}
{\sc F{\'e}lix, Y., Halperin, S., and Thomas, J.-C.}
\newblock Gorenstein spaces.
\newblock {\em Adv. in Math. 71}, 1 (1988), 92--112.

\bibitem{FHT}
{\sc F{\'e}lix, Y., Halperin, S., and Thomas, J.-C.}
\newblock {\em Rational homotopy theory}, vol.~205 of {\em Graduate Texts in
  Mathematics}.
\newblock Springer-Verlag, New York, 2001.

\bibitem{FIJ}
{\sc Frankild, A., Iyengar, S., and J{\o}rgensen, P.}
\newblock Dualizing differential graded modules and {G}orenstein differential
  graded algebras.
\newblock {\em J. London Math. Soc. (2) 68}, 2 (2003), 288--306.

\bibitem{FJ2}
{\sc Frankild, A., and J{\o}rgensen, P.}
\newblock Homological identities for differential graded algebras {II}.
\newblock PhD-thesis A. Frankild, 2002.

\bibitem{FJ}
{\sc Frankild, A., and J{\o}rgensen, P.}
\newblock Gorenstein differential graded algebras.
\newblock {\em Israel J. Math. 135\/} (2003), 327--353.

\bibitem{H1}
{\sc Happel, D.}
\newblock On the derived category of a finite-dimensional algebra.
\newblock {\em Comment. Math. Helv. 62}, 3 (1987), 339--389.

\bibitem{H}
{\sc Happel, D.}
\newblock {\em Triangulated categories in the representation theory of
  finite-dimensional algebras}, vol.~119 of {\em London Mathematical Society
  Lecture Note Series}.
\newblock Cambridge University Press, Cambridge, 1988.

\bibitem{H2}
{\sc Happel, D.}
\newblock Auslander-{R}eiten triangles in derived categories of
  finite-dimensional algebras.
\newblock {\em Proc. Amer. Math. Soc. 112}, 3 (1991), 641--648.

\bibitem{HKR}
{\sc Happel, D., Keller, B., and Reiten, I.}
\newblock Bounded derived categories and repetitive algebras.
\newblock arXiv:math/0702302v1 [math.RT], 2007.

\bibitem{HPR}
{\sc Happel, D., Preiser, U., and Ringel, C.~M.}
\newblock Vinberg's characterization of {D}ynkin diagrams using subadditive
  functions with application to {$D{\rm Tr}$}-periodic modules.
\newblock In {\em Representation theory, II (Proc. Second Internat. Conf.,
  Carleton Univ., Ottawa, Ont., 1979)}, vol.~832 of {\em Lecture Notes in
  Math.} Springer, Berlin, 1980, pp.~280--294.

\bibitem{Hat}
{\sc Hatcher, A.}
\newblock {\em Algebraic topology}.
\newblock Cambridge University Press, Cambridge, 2002.

\bibitem{Hess}
{\sc Hess, K.}
\newblock Rational homotopy theory: a brief introduction.
\newblock In {\em Interactions between homotopy theory and algebra (Chicago,
  2004)}, vol.~436 of {\em Contemporary Mathematics}. American Mathematical
  Society, 2007, pp.~175--202.

\bibitem{J1}
{\sc J{\o}rgensen, P.}
\newblock Auslander-{R}eiten theory over topological spaces.
\newblock {\em Comment. Math. Helv. 79}, 1 (2004), 160--182.

\bibitem{J3}
{\sc J{\o}rgensen, P.}
\newblock Amplitude inequalities for differential graded modules.
\newblock arXiv:math/0601416v1 [math.RA], 2006.

\bibitem{J2}
{\sc J{\o}rgensen, P.}
\newblock The {A}uslander-{R}eiten quiver of a {P}oincar\'e duality space.
\newblock {\em Algebr. Represent. Theory 9}, 4 (2006), 323--336.

\bibitem{Ke}
{\sc Keller, B.}
\newblock Deriving {DG} categories.
\newblock {\em Ann. Sci. \'Ecole Norm. Sup. (4) 27}, 1 (1994), 63--102.

\bibitem{K0}
{\sc Krause, H.}
\newblock Auslander-{R}eiten theory via {B}rown representability.
\newblock {\em $K$-Theory 20}, 4 (2000), 331--344.
\newblock Special issues dedicated to Daniel Quillen on the occasion of his
  sixtieth birthday, Part IV.

\bibitem{K}
{\sc Krause, H.}
\newblock A {B}rown representability theorem via coherent functors.
\newblock {\em Topology 41}, 4 (2002), 853--861.

\bibitem{K1}
{\sc Krause, H.}
\newblock Auslander-{R}eiten triangles and a theorem of {Z}immermann.
\newblock {\em Bull. London Math. Soc. 37}, 3 (2005), 361--372.

\bibitem{K2}
{\sc Krause, H.}
\newblock Derived categories, resolutions, and {B}rown representability.
\newblock In {\em Interactions between homotopy theory and algebra (Chicago,
  2004)}, vol.~436 of {\em Contemporary Mathematics}. American Mathematical
  Society, 2007, pp.~101--139.

\bibitem{N2}
{\sc Neeman, A.}
\newblock The connection between the {$K$}-theory localization theorem of
  {T}homason, {T}robaugh and {Y}ao and the smashing subcategories of
  {B}ousfield and {R}avenel.
\newblock {\em Ann. Sci. \'Ecole Norm. Sup. (4) 25}, 5 (1992), 547--566.

\bibitem{N}
{\sc Neeman, A.}
\newblock {\em Triangulated categories}, vol.~148 of {\em Annals of Mathematics
  Studies}.
\newblock Princeton University Press, Princeton, NJ, 2001.

\bibitem{Op}
{\sc Oppermann, S.}
\newblock Lower bounds for {A}uslander's representation dimension.
\newblock Preprint, 2007.

\bibitem{RV}
{\sc Reiten, I., and Van~den Bergh, M.}
\newblock Noetherian hereditary abelian categories satisfying {S}erre duality.
\newblock {\em J. Amer. Math. Soc. 15}, 2 (2002), 295--366 (electronic).

\bibitem{Rq}
{\sc Rouquier, R.}
\newblock Dimensions of triangulated categories.
\newblock arXiv:math/0310134v3 [math.CT], 2003.

\bibitem{S}
{\sc Spaltenstein, N.}
\newblock Resolutions of unbounded complexes.
\newblock {\em Compositio Math. 65}, 2 (1988), 121--154.

\bibitem{Spanier}
{\sc Spanier, E.~H.}
\newblock {\em Algebraic topology}.
\newblock McGraw-Hill Book Co., New York, 1966.

\bibitem{V}
{\sc Verdier, J.-L.}
\newblock Des cat\'egories d\'eriv\'ees des cat\'egories ab\'eliennes.
\newblock {\em Ast\'erisque}, 239 (1996), xii+253 pp. (1997).
\newblock With a preface by Luc Illusie, Edited and with a note by Georges
  Maltsiniotis.

\end{thebibliography}

\end{document}